\documentclass[a4paper,12pt]{article}

\usepackage{bbm}
\usepackage{amsmath, wrapfig}
\usepackage{dsfont, a4wide, amsthm, amssymb, amsfonts, graphicx}
\usepackage{fancyhdr, xspace, psfrag, setspace, supertabular, color, hyperref}

\newtheorem{theorem}{Theorem}[section]
\newtheorem{proposition}[theorem]{Proposition}
\newtheorem{lemma}[theorem]{Lemma}

\newtheorem{corollary}[theorem]{Corollary}
\newtheorem{conjecture}[theorem]{Conjecture}
\newtheorem{definition}[theorem]{Definition}

\newtheorem{question}[theorem]{Question}

\theoremstyle{remark}
\newtheorem*{remark}{Remark}

\def\a{\alpha}
\def\d{\delta}
\def\e{\epsilon}

\def\Z{\mathbb{Z}}
\def\F{\mathbb{F}}

\def\R{\mathbb{R}}
\def\C{\mathbb{C}}
\def\E{\mathbb{E}}
\def\P{\mathbb{P}}
\def\N{\mathbb {N}}

\DeclareMathOperator{\rk}{rk}
\DeclareMathOperator{\erk}{erk}

\DeclareMathOperator{\tr}{tr}

\DeclareMathOperator{\pr}{pr}
\DeclareMathOperator{\epr}{epr}
\DeclareMathOperator{\dpr}{dpr}

\DeclareMathOperator{\ar}{ar}
\DeclareMathOperator{\bias}{bias}

\DeclareMathOperator{\supp}{supp}

\newcommand{\new}{\mathrm{new}}

\title{Equidistribution of high-rank polynomials with variables restricted to subsets of $\F_p$}
\author{W. T. Gowers \footnote{Coll\`ege de France and Department of Pure Mathematics and Mathematical Statistics, University of Cambridge. Email: \texttt{wtg10@dpmms.cam.ac.uk}.} \and Thomas Karam \footnote{Department of Pure Mathematics and Mathematical Statistics, University of Cambridge. Email: \texttt{tk542@cam.ac.uk}.}}

\begin{document}
\maketitle

\begin{abstract}

Let $p$ be a prime and let $S$ be a non-empty subset of $\F_p$. Generalizing a result of Green and Tao on the equidistribution of high-rank polynomials over finite fields, we show that if $P: \F_p^n \rightarrow \F_p$ is a polynomial and its restriction to $S^n$ does not take each value with approximately the same frequency, then there exists a polynomial $P_0: \F_p^n \rightarrow \F_p$ that vanishes on $S^n$, such that the polynomial $P-P_0$ has bounded rank. Our argument uses two black boxes: that a tensor with high partition rank has high analytic rank and that a tensor with high essential partition rank has high disjoint partition rank.

\end{abstract}

\tableofcontents

\section{Introduction}

Our starting point in this paper is the following theorem of Green and Tao \cite[Theorem 1.7]{Green and Tao}, which broadly speaking states that if a multivariable polynomial over $\F_p$ does not take each value with approximately the same frequency, then it can be expressed in terms of a bounded number of polynomials of lower degree. We write $\omega_p$ for $\exp(2\pi i/p)$.

\begin{theorem}\label{Theorem of Green and Tao} Let $p \ge 2$ be a prime and let $d$ be an integer such that $0 < d < p$. Then there exists a function $B_{p, d}: (0, 1 \rbrack \rightarrow \lbrack 0, \infty )$ such that for every $\epsilon>0$, if $P: \F_p^n \rightarrow \F_p$ is a polynomial with degree $d$ such that \begin{equation} |\E_{x \in \F_p^n} \omega_p^{P(x)}| \ge \epsilon \label{high bias of a polynomial} \end{equation} then there exists $k \le B_{p,d}(\epsilon)$, polynomials $P_1, \dots, P_k: \F_p^n \rightarrow \F_p$ each with degree at most $d-1$ and a function $F: \F_p^k \rightarrow \F_p$ such that $P = F(P_1, \dots, P_k)$. \end{theorem}

The assumption $d<p$ was later removed \cite{Kaufman and Lovett} by Kaufman and Lovett. The quantity $\E_{x \in \F_p^n} \omega_p^{P(x)}$ from Theorem \ref{Theorem of Green and Tao} is often referred to as the \emph{bias} of the polynomial $P$, and the smallest possible nonnegative integer $k$ in its conclusion is often called the \emph{rank} of $P$. 

Note that for any $t\in\F_p$ and any function $f:\F_p^n\to\F_p$,
\[\P[f(x)=t]=\E_{x\in\F_p^n}\mathbbm 1_{f(x)=t}=\E_{s\in\F_p}\E_{x\in\F_p^n}\omega_p^{s(f(x)-t)}=p^{-1}+\E_{s\ne 0}\E_{x\in\F_p^n}\omega_p^{s(f(x)-t)}.\]
It follows that if some value is taken by $f$ with a probability that differs from $p^{-1}$ by at least $\d$, then there exists $s\ne 0$ such that $|\E_{x\in\F_p^n}\omega^{sf(x)}|\geq\d$. In particular, if the values of a polynomial $P$ are not approximately uniformly distributed, then some non-zero multiple of $P$ has large bias.

As we shall discuss in more detail later in this introduction, it has subsequently been shown that the function $F$ that appears in the statement of Theorem \ref{Theorem of Green and Tao} can be taken to be of a particular form: under the assumptions of the theorem, there are polynomials $Q_1,\dots,Q_k$ and $R_1,\dots,R_k$ of degree at most $d-1$ such that $P=Q_1R_1+\dots+Q_kR_k$. 

Our aim in this paper is to generalize Theorem \ref{Theorem of Green and Tao} to a statement concerning a restricted alphabet. Let $S$ be a proper subset of $\F_p$. We shall say that $P:\F_p^n\to\F_p$ is biased on $S^n$ if it does not take each value with approximately the same frequency. We wish to formulate and prove a statement to the effect that if $P$ is biased on $S^n$, then it has low rank on $S^n$.

Proving results for restricted alphabets is already a mini-theme in additive combinatorics. For example, Bourgain, Dilworth, Ford, Konyagin and Kutzarova have proved a lower bound \cite[Theorem 5]{BDFKK} for the size of the sumset $A+B$ in $\Z^n$ if $A$ and $B$ are both subsets of $\{0,1,\dots,M-1\}^n$ (which has obvious consequences for subsets of $S^n\subset\F_p^n$ if $S$ is Freiman isomorphic to a subset of $\Z$). There has also been work \cite{DGIM, Kane and Tao} on minimizing the additive energy of a subset of $\{0,1\}^n$ of given density. 

For us the motivation for considering restricted alphabets came from density Hales-Jewett type problems. We shall not explain it in full here, but suppose, for example, that one wishes to find conditions on three sets $A,B,C\subset\{0,1,2\}^n$ that will ensure that there is a combinatorial line $(x,y,z)$ with $x\in A, y\in B$ and $z\in C$. (This is a set of points such that for each coordinate $i$, the triple $(x_i,y_i,z_i)$ is equal to $(0,0,0), (1,1,1), (2,2,2)$, or $(0,1,2)$.) If we write $|x|$ for the sum of the coordinates of $x$ (in $\Z$), then we see that $|x|, |y|$ and $|z|$ form an arithmetic progression, so we can define sets such as $A=B=\{x:|x|\equiv 0\ (\text{mod}\ 5)\}$ and $C=\{z:|z|\equiv 1\ \text{mod}\ 5\}$. In this way, it becomes natural to regard $\{0,1,2\}$ as a subset of $\F_5$ and consider the restriction of the linear form $|x|$ to $\{0,1,2\}^n$. And for longer combinatorial lines it becomes natural in a similar way to look at restrictions of polynomials of higher degree. 

In particular, one can use polynomials over $\F_p$ to define ``Bohr-like" sets on $\{0,1,\dots,k-1\}^n$. First, one regards $\{0,1,\dots,k-1\}$ as a subset of $\F_p$. Then given polynomial functions $P_1,\dots,P_r:\F_p^n\to\F_p^t$ for some $t$, and subsets $E_1,\dots,E_r$ of $\F_p^t$, one can define a set $B((P_1,E_1),\dots,(P_r,E_r))$ to be the set of all $x\in\{0,1,\dots,k-1\}^n$ such that $P_i(x)\in E_i$ for each $i$. Let us (just for this paragraph) call this an \emph{$\F_p$-polynomial Bohr set in $[k]^n$ with parameters} $(t,r)$. In a separate paper \cite{Gowers and K.}, we shall use the results of this paper to prove that every $\F_p$-polynomial Bohr set in $[k]^n$ with parameters $(t,r)$ can be approximated to within a density $\d$ by an $\F_p$-polynomial Bohr set in $[k]^n$ with parameters $(t',r')$, where $t'$ and $r'$ are bounded above by functions of $\d, p, k$ and $t$. (The main point here is that the upper bounds for $t'$ and $r'$ do not depend on $r$, so if a dense set is defined by a large number of polynomial conditions, then it can be approximated by a set defined by a bounded number of polynomial conditions.)

To understand the formulation of our main result, it will help to have examples such as the following in mind. Let $P(x)=\prod_{i=1}^n x_i(x_i-1)$. Then $P$, considered as a polynomial on $\F_p^n$, has very high rank and is consequently very close to being equidistributed. However, its restriction to $\{0,1\}^n$ is identically zero, so is as unevenly distributed as possible. We would like to be able to say that this restriction has low rank in some sense. 

We do this in a fairly obvious way: we say that $P$ has low rank on $S^n$ if there is a low-rank polynomial $Q$ on $\F_p^n$ such that $P-Q$ vanishes on $S^n$. For instance, in the example above, we can simply take $Q$ to be the zero polynomial. 

We now make this precise. We start by defining appropriate versions of bias and rank in the context of restricted alphabets. 

\begin{definition} \label{bias of a function} Let $\F$ be a finite field, let $S$ be a subset $\F$, and let $f: \F^n \rightarrow \F$ be a function. For $\chi: \F \rightarrow \C$ a non-trivial character we define the \emph{bias of the function $f$ with respect to the character $\chi$} and the set $S$ by
\[\bias_{\chi,S}(f)=\E_{x\in S^n}\chi(f(x)).\]
More generally, if $D$ is a probability distribution on $\F_p$, we define the bias of $F$ with respect to $\chi$ and $D$ by
\[\bias_{\chi,D}(f)=\E_{x\sim D^n}\chi(f(x)).\]
When $S=\F$ (or equivalently $D$ is the uniform distribution on $\F$), we write $\bias_{\chi}(f)$ for $\bias_{\chi,S}(f)$.\end{definition}

When $\F=\F_p$ the non-trivial characters are the functions $\chi_t: x \mapsto \omega_p^{tx}$ for each $t \in \F_p^*$. In this case we write $\bias_{t,S} f$ for $\bias_{\chi_t,S} f$.

\begin{definition} \label{rank of a polynomial} Let $\F$ be a field, let $d$ be a non-negative integer, and let $P: \F^n \rightarrow \F$ be a polynomial of degree at most $d$. The \emph{degree-$d$ rank} of $P$, denoted by $\rk_d(P)$, is defined as follows.

\begin{enumerate}
\item If $d=0$, then $\rk_d P = 0$.
\item If $d=1$, then $\rk_d P$ is the number of $i$ such that $a_i\ne 0$ in the unique representation of $P$ in the form  $P(x)=c+\sum_{i=1}^na_ix_i$, where $c,a_1,\dots,a_n\in\F_p$. 
\item If $d\ge 2$, then $\rk_d P$ is the smallest nonnegative integer $k$ such that there exist polynomials $Q_1, R_1, \dots, Q_k, R_k: \F^n \rightarrow \F$, each with degree strictly smaller than $d$, and with $\deg(Q_i)+\deg(R_i)\leq d$ for each $i$, such that \begin{equation}P = Q_1 R_1 + \dots + Q_k R_k. \label{rank decomposition of a polynomial} \end{equation}
\end{enumerate}

\noindent If $S$ is a non-empty subset of $\F$, then the \emph{degree-$d$ rank of $P$ with respect to $S$}, denoted by $\rk_{d,S}P$, is $\min_{P_0} \rk_d (P - P_0)$, where the minimum is taken over all polynomials $P_0: \F^n \rightarrow \F$ such that $P_0(S^n) = \{0\}$. Equivalently, it is the minimum degree-$d$ rank of any polynomial $Q$ that agrees with $P$ on $S^n$.

If $P$ has degree exactly $d$, then we define the \emph{rank} of $P$ to be $\rk_d(P)$, and denote it by $\rk P$.

Finally, if $S$ is a non-empty subset of $\F$, then we define $\rk_SP$ to be the minimum of $\rk_{\deg P'}(P')$ over all polynomials $P'$ that agree with $P$ on $S^n$. 
\end{definition}

When $S=\F_p$, this definition is very similar to that of Green and Tao. One difference is that in the case $d\geq 2$, we ask for a decomposition $P=Q_1 R_1 + \dots + Q_k R_k$ rather than one of the form $P=F(P_1,\dots,P_k)$. This arises naturally from the notion of partition rank, as we shall explain later in this introduction.

A second difference is our definition when $d=1$, which may at first seem a little strange, since there is a readily available notion of rank coming from linear algebra, according to which the degree-1 rank of $P$ would be 1 if $P$ is non-constant and 0 if $P$ is constant. If $P$ is defined on all of $\F_p^n$, then this is a satisfactory definition, since then the values of $P$ will be exactly equidistributed if $P$ is non-constant, and very unevenly distributed if $P$ is constant. However, these simple facts clearly do not continue to hold when $P$ is defined on $S^n$ for some proper subset $S\subset\F_p$. For example, if $P(x)=x_1$, then the values of $P$ are all contained in $S$. However, a fairly simple Fourier-analysis argument can be used to show that if $P$ has large degree-1 rank in the sense defined above, then its restriction to $S^n$ \emph{will} be approximately equidistributed (provided $|S|\geq 2$).

Note also that when $d=2$ the definition is similar, but not identical, to the usual definition of the rank of a quadratic form as the rank of the associated symmetric bilinear form. For example, the rank of the form $x\mapsto x_1x_2$ is obviously 1 when defined as above, but the associated symmetric bilinear form has rank 2 (over a field of odd characteristic).

The definition of the rank of a polynomial is slightly unnatural in that it is not subadditive, since if $P$ and $Q$ are low-rank polynomials of degree $d$, it may be that $P+Q$ is a high-rank polynomial of degree less than $d$. However, for certain statements it is nevertheless a convenient definition to have. One reason for this is that if $S$ is a proper subset of $\F_p$, then a polynomial $P$ of degree $d$ on $\F_p^n$ can agree on $S^n$ with a polynomial $Q$ of lower degree. In such a situation, the rank of $Q$ affects how well $P$ is equidistributed on $S^n$. The definition we have given allows us to express this in a concise way.
\medskip

We now state our main theorem. Like Green and Tao, we assume that the field size is greater than the degree, but it is likely that the assumption is not necessary.

\begin{theorem} \label{Main theorem} Let $d$ be a positive integer and let $p>d$ be a prime. Let $ S$ be a non-empty finite subset of $ \F_p$, and let $ P$ be a polynomial of degree $ d$. Then there exists a function $H_{p,d,S}: (0, 1 \rbrack \rightarrow \lbrack 0, + \infty )$ such that if $P: \F^n \rightarrow \F$ is a polynomial with degree $d$ such that there exists a non-trivial character $\chi: \F_p \rightarrow \C$ for which $|\bias_{\chi,S} P| \ge \epsilon$, then $ \rk_{d,S} P \le \rk_SP \le H_{p,d,S}(\epsilon)$. \end{theorem}

\subsection{Background results on tensors}

In our proof, we shall appeal to two known results, which we shall use as black boxes. The first is an analogue of Theorem \ref{Theorem of Green and Tao} for tensors, and the second is a result of the second author that allows us, under suitable conditions, to restrict a high-rank tensor to a product of disjoint sets in such a way that it remains of high rank. This ``decoupling" will play an important role in the reduction of the polynomial statement to the tensor statement.

By an \emph{order-$d$ tensor} over a field $\F$ we mean simply a function $T: X_1\times\dots\times X_d \rightarrow \F$ for some finite sets $X_1, \dots, X_d$. If $T$ is an order-$d$ tensor, we can associate with it a $d$-linear form $m:\F^{X_1}\times\dots\times\F^{X_d}\to\F$, defined by the formula \[m(x^1, \dots, x^d) = \sum_{(i_1, \dots, i_d) \in X_1\times\dots\times X_d} T(i_1, \dots, i_d) x^1_{i_1} \dots x^d_{i_d}. \]

The analogues for tensors of the bias and of rank of polynomials that will be relevant to us are respectively the analytic rank introduced by Gowers and Wolf in \cite{Gowers and Wolf} and the partition rank introduced by Naslund in \cite{Naslund}.

\begin{definition}\label{Analytic rank}

Let $d \ge 2$ be a positive integer, let $\F$ be a finite field and let $T: X_1\times\dots\times X_d \rightarrow \F$ be an order-$d$ tensor, and let $m$ be the $d$-linear form associated with $T$. The \emph{bias of $m$} is defined to be 
\begin{equation} \bias m = \E_{x_1 \in \F^{X_1}, \dots, x_d \in \F^{X_d}} \chi(m(x_1, \dots, x_d)) \label{expression of bias for a multilinear form} \end{equation} for any arbitrary non-trivial character $\chi: \F \rightarrow \C$ of $\F$. The \emph{analytic rank} of $T$, denoted by $\ar T$, is defined to be $- \log_{|F|} \bias m$,\end{definition}

The right-hand side of \eqref{expression of bias for a multilinear form} is independent of the non-trivial character $\chi$. Indeed, the only property we require of $\chi$ is that if $\lambda:\F\to\F$ is a linear map, then $\E_x\chi(\lambda(x))=1$ if $\lambda$ is the zero map and 0 otherwise. From this property, it follows easily that $\bias(T)$ is always a positive real number: indeed, it is equal to the probability that if we randomly restrict $x_1,\dots,x_{d-1}$, then the resulting linear map from $\F^{X_d}$ to $\F$ is identically zero. (Of course, we could randomly restrict any $d-1$ of the coordinates and get the same result.) However, the definition in terms of characters is more convenient for the purposes of generalization.

In the next definition we write $[d]$ for the set $\{1,2,\dots,d\}$, as is customary. If $x=(x_1,\dots,x_d)\in X_1\times\dots\times X_d$ and $I\subset[d]$, then we write $x(I)$ for the restriction of $x$ to $I$ -- that is, for the element $y$ of $\prod_{i\in I}X_i$ such that $y_i=x_i$ for each $i\in I$. 

\begin{definition} \label{Partition rank}

Let $d \ge 2$ be a positive integer, let $\F$ be a field and let $T: X_1\times\dots\times X_d \rightarrow \F$ be an order-$d$ tensor. Then $T$ has \emph{partition rank at most 1} if there exists a non-trivial partition of $[d]$ into sets $I,J$, and functions $a:\prod_{i\in I}X_i\to\F$ and $b:\prod_{i\in J}X_i\to\F$, such that $T(x)=a(x(I))b(x(J))$ for every $x\in X_1\times\dots\times X_d$. The \emph{partition rank} of $T$, denoted $\pr T$, is the smallest non-negative integer $k$ such that $T$ is a sum of $k$ tensors of partition rank at most 1. If $m$ is the multilinear form associated with $T$, then the partition rank $\pr m$ of $m$ is defined to be $\pr T$. 
\end{definition}

Since the paper \cite{Green and Tao} of Green and Tao there has been significant interest in comparing these two notions of rank. Our first black box is that when $\F$ and $d$ are fixed, an order-$d$ tensor over $\F_p$ with large partition rank necessarily has a large analytic rank.

\begin{theorem} \label{bounded analytic rank implies bounded partition rank} Let $d \ge 2$ be a positive integer and let $\F$ be a finite field. Then there exists a function $A_{d,\F}: \lbrack 0, \infty) \rightarrow \lbrack 0, \infty)$ such that $\pr T \le A_{d,\F}(\ar T)$ for every order-$d$ tensor $T$ over $\F$. \end{theorem}

In the regime where $\F$ has fixed size, which is the case that we will consider through this paper, the best bounds for the function $A_{d,\F}$ are due to Janzer \cite{Janzer} and Mili\'cevi\'c \cite{Milicevic}. For each $r\ge 1$, Janzer obtains \cite[Theorem 1.10]{Janzer} the bound $A_{d,\F}(r) = (c \log |\mathbb{F}|)^{c'(d)} (r)^{c'(d)}$ for $c$ an absolute constant and $c'(d) = 4^{d^d}$ and Mili\'cevi\'c obtains \cite[Theorem 3]{Milicevic} the bound $A_{d,\F}(r) = 2^{d^{2^{O(d^2)}}} (r^{2^{2^{O(d^2)}}} + 1)$. 

We shall deduce from Theorem \ref{bounded analytic rank implies bounded partition rank} a similar statement for restricted alphabets. Let $d$ be a positive integer and let $p$ be a prime, let $X_1,\dots,X_d$ be finite sets, let $m: (\F^{X_1} \times \dots \times \F^{X_d}) \rightarrow \F$ be a $d$-linear form, and let $\chi: \F \rightarrow \C$ be a non-trivial character. Then given non-empty subsets $S_1,\dots,S_d$ of $\F_p$, we define 
\[\bias_{\chi, (S_1, \dots, S_d)} m:=\E_{(x_1, \dots, x_d) \in S_1^{X_1} \times \dots \times S_d^{X_d}} \chi(m(x_1, \dots, x_d)).\]
Similarly, if $D_1,\dots,D_d$ are probability distributions on $\F_p$, then we define
\[\bias_{\chi, (D_1, \dots, D_d)} m:= \E_{(x_1, \dots, x_d) \sim D_1^{X_1} \times \dots \times D_d^{X_d}} \chi(m(x_1, \dots, x_d)).\] 
Here $D_i^{X_i}$ stands for the distribution on $\F^{X_i}$ where each coordinate is chosen independently according to the distribution $D_i$. 

Note that it is no longer the case that the bias must be a positive real number or that it is independent of the choice of $\chi$. However, we do not need these properties in the formulation of the next result.

\begin{proposition} \label{bounded analytic rank implies bounded partition rank for restricted alphabets} Let $p$ be a prime, let $d \ge 2$ and let $S_1, \dots, S_d$ be subsets of $\F_p$ each with size at least $2$. Then there exists a function $A_{d,p, (S_1, \dots, S_d)}: \lbrack 0, \infty) \rightarrow \lbrack 0, \infty)$ such that whenever $T$ is an order-$d$ tensor and the $d$-linear form $m$ associated with $T$ satisfies
\[ |\bias_{t, (S_1,\dots,S_d)} m| \ge \epsilon \] for some $t \in \F_p^*$, then $\pr T \le A_{d,p, (S_1, \dots, S_d)}(-\log_{p}(\epsilon))$. \end{proposition}

In the special case $p=2$ the statement of Proposition \ref{bounded analytic rank implies bounded partition rank for restricted alphabets} is the same as that of Theorem \ref{bounded analytic rank implies bounded partition rank}. We shall therefore assume that $p \ge 3$ in the remainder of the paper. We also remark here that two of our later lemmas, Lemma \ref{Writing a distribution as a convolution with 0,1 for F_p} and a generalization of it, Lemma \ref{Writing a distribution as a convolution with 0,1}, do not hold when $p=2$.

The reader will notice that we have not formulated a notion of partition rank for restricted alphabets, and that the conclusion of Proposition \ref{bounded analytic rank implies bounded partition rank for restricted alphabets} concerns the usual partition rank of $T$. The reason is that a multilinear map defined on $\F^{X_1}\times\dots\times\F^{X_d}$ is determined by its values on $S_1^{X_1}\times\dots\times S_d^{X_d}$. This is easy to prove inductively: a linear map defined on $\F^{X_1}$ is determined by its values on $S_1^{X_1}$ (since $S_1^{X_1}$ spans $\F^{X_1}$), so for each $x_1,\dots,x_{d-1}\in S_1^{X_1}\times\dots\times S_{d-1}^{X_{d-1}}$, if two multilinear maps agree on $\{x_1\}\times\dots\times\{x_{d-1}\}\times S_d^{X_d}$, then they agree on $\{x_1\}\times\dots\times\{x_{d-1}\}\times\F^{X_d}$. Repeating for each coordinate, we get that they agree everywhere. (Here we made use of the assumption that each $S_i$ has size at least 2.) 

As Janzer and Mili\'cevi\'c explain, we can deduce Theorem \ref{Theorem of Green and Tao} from Theorem \ref{bounded analytic rank implies bounded partition rank}, using a connection with uniformity norms. It is known that the bias of a function $f$ is bounded above in modulus by its uniformity norm $\|f\|_{U^k}$ for any $k$. But if $f$ is of the form $\chi\circ P$ for a polynomial $P$ of degree $d$, then $\|f\|_{U^d}$ can be shown to equal $\bias(m)$, where $m$ is a multilinear form naturally associated with $P$ that satisfies an identity of the form $P(x)=(d!)^{-1}m(x,x,\dots,x)+W(x)$ for a polynomial $W$ of degree strictly less than $d$. Since $m$ has large bias, it has small partition rank, and this translates into an upper bound for the degree-$d$ rank of $P$.

However, we cannot use this method to deduce Theorem \ref{Main theorem} from Proposition \ref{bounded analytic rank implies bounded partition rank for restricted alphabets} for multilinear forms, because it relies on several identities that do not apply when we restrict the alphabet. For example, $\E_{x,y \in \F_p} f(x) \overline{f(x+y)}=|\E_{x \in \F_p} f(x)|^2$, but $\E_{x,y \in S} f(x) \overline{f(x+y)}$ is not related in any simple way to $|\E_{x \in S} f(x)|^2$.

Instead, we shall combine a slight generalization of Proposition \ref{bounded analytic rank implies bounded partition rank for restricted alphabets} with the second black box, the purpose of which is to decouple the variables of a polynomial, so that polynomials are reduced to multilinear forms in our arguments. 

To see roughly how this will work, consider the quadratic case. If $P(x)=\sum_{i,j=1}^na_{ij}x_ix_j$ (for simplicity we look at the homogeneous case here, but we shall prove the result in general), then we can write it in the form $\sum_ia_{ii}x_i^2+\sum_{i\ne j}a_{i,j}x_ix_j$. Consider now the bilinear form $\beta(x,y)=\sum_{i\ne j}a_{ij}x_iy_j$. If $\beta+\delta$ has high rank for every diagonal bilinear form $\d$ (that is, one given by a formula $\d(x)=\sum_ic_ix_i^2$), then our second black box (which is a simple result in the quadratic case, but much less so for higher degree) allows us to conclude that there is a partition of $[n]$ into sets $X$ and $Y$ such that the restriction $\beta'$ of $\beta$ to $\F_p^X\times\F_p^Y$ still has high rank. If we now regard $P$ as a 2-variable function defined on $\F_p^X\times\F_p^Y$, a standard argument bounds the $t$-bias of $P$ by the box norm of the function $\omega_p^{tP}$, which is equal to the expectation $\E_{x,y}\omega_p^{t\beta'(x,y)}$. Since $\beta'$ has high rank, this expectation is small, so $P$ has small $t$-bias for every $t\in\F_p^*$. Taking the contrapositive, if we assume that $P$ has large bias for some $t$, then we deduce that $\beta'$ also has large bias, and therefore that it has small partition rank, which in this case is just the rank, which implies that $\beta+\d$ has small rank for some diagonal bilinear form $\d$, by how we chose $\beta'$. This allows us to write $P(x,y)$ as the sum of a low-rank quadratic form and a diagonal form. It is then not too hard to prove that for $P$ to have large bias, the diagonal form must have small support.

We have not mentioned the set $S$ in the above sketch, but it turns out that, unlike with the proof based on uniformity norms, all the elements of the argument carry over for restricted alphabets, as we shall see in the next section. 
\medskip

To complete this section, we give a precise statement of the theorem that forms the second black box. Let $d \ge 2$ and $n_1,\dots,n_d$ be positive integers and for each $i$ let $X_i$ be a subsets of $[n_i]$.  Given an order-$d$ tensor $T: \lbrack n_1 \rbrack \times \dots \times \lbrack n_d \rbrack \rightarrow \mathbb{F}$ we write $T(X_1 \times \dots \times X_d)$ for the restriction of $T$ to $X_1 \times \dots \times X_d$. Similarly, given a multilinear form $m: \F^{n_1} \times \dots \times \F^{n_d} \rightarrow \mathbb{F}$ a $d$-linear form we write $m(\F^{X_1} \times \dots \times \F^{X_d})$ for the restriction of $m$ to $\F^{X_1} \times \dots \times \F^{X_d}$.

\begin{definition} \label{Essential and partition disjoint rank definition} Let $d \ge 2$ be an integer, let $\F$ be a finite field, and let $T: [n]^d \rightarrow \mathbb{F}$ be an order-$d$ tensor. Let $E\subset[n]^d$ be the set of all $d$-tuples with at least one pair of equal coordinates. The \emph{essential partition rank} of $T$ is the quantity
\[ \epr T = \min_V \pr (T+V), \] 
where the minimum is taken over all order-$d$ tensors $V: [n]^d \rightarrow \F$ that are supported in $E$. The \emph{disjoint partition rank} is the quantity 
\[ \dpr T= \max_{X_1,\dots ,X_d} \pr (T(X_1 \times \dots \times X_d)), \] 
where the maximum is taken over all sequences of disjoint subsets $X_1,\dots,X_d$ of $[n]$.
\end{definition}

Since restricting a tensor cannot increase its partition rank, and since adding a tensor supported in $E$ has no effect on the disjoint partition rank (because $X_1\times\dots\times X_d$ and $E$ are disjoint if the $X_i$ are disjoint), we see immediately that the disjoint partition rank is bounded above by the essential partition rank. The result we shall use, due to the second author \cite{K.}, provides a bound in the other direction.

\begin{theorem}\label{Disjoint rank minors theorem} For every positive integer $d \ge 2$ there exists a function $\Lambda_d: \mathbb{N} \rightarrow \mathbb{N}$ such that if $T$ is an order-$d$ tensor satisfying $\epr T \ge \Lambda_d(l)$ then $\dpr T \ge l$. \end{theorem}

The rest of the paper is organized as follows. In Section \ref{Section: The linear, bilinear and quadratic cases} we shall prove Theorem \ref{Main theorem} in the cases $d=1$ and $d=2$. In Section \ref{Section: The proof in the case of an alphabet with size 2} we shall then prove Theorem \ref{Main theorem} in the case of alphabets $S$ of size 2, where the proof is significantly simpler than in the general case. After this, in Section \ref{Section: Equidistribution of multi-alphabet multilinear forms} we shall state and prove a more complicated version of Proposition \ref{bounded analytic rank implies bounded partition rank for restricted alphabets}, which, together with Theorem \ref{Disjoint rank minors theorem}, will then allow us to deduce Theorem \ref{Main theorem}. This we shall do in Section \ref{Section: The general polynomial case}.

\section{The linear, bilinear and quadratic cases} \label{Section: The linear, bilinear and quadratic cases}

\subsection{The linear case}

If $l: \F_p^n \rightarrow \F_p$ is a linear form, then $\bias l$ is equal to $1$ if $l \equiv 0$, and equal to $0$ otherwise. As mentioned in the introduction, in the case of restricted subsets of $\F_p$, the modulus of the bias of $l$ instead decreases, as it turns out exponentially, with the order-1 rank of $l$, which we defined earlier to be the size of the support of $l$ when it is considered as a vector in $\F_p^n$ in the usual way.

We begin with a technical lemma.

\begin{lemma} \label{average of roots of unity}
Let $p$ be a prime, let $0<c\leq 1/2$ be a real number, and let $D$ be a probability distribution on $\F_p$ such that $D(x)\le 1-c$ for every $x\in\F_p$. Then $|\E_{x\sim D}\omega_p^x|\leq 1-c\pi^2/p^2$.
\end{lemma}

\begin{proof}
Let $\a=\|D\|_\infty$ (that is, $\max_xD(x)$). We claim first that $D$ is a convex combination of distributions $D'$ with $\|D'\|\leq\max\{\a,1/2\}$ and of support size at most 2.

If $\a\geq 1/2$, then let $x$ be such that $D(x)=\a$. Then for each $y$ let $D_y(x)=\a$, let $D_y(y)=1-\a$, and let $D_y(z)=0$ for all other $z$. Then $D=(1-\a)^{-1}\sum_yD(y)D_y$.

If $\a<1/2$, then observe first that every distribution that is uniform on a subset $A$ of $\F_p$ of size at least 2 is a convex combination of distributions of the required form. (These are now distributions that are uniform on a two-point subset -- just take the average over all two-point subsets of $A$.) Now let $x$ be such that $D(x)=\a$ and let $A$ be the support of $D$. Then subtract a multiple of $\mathbbm 1_A$ from $D$ to obtain a non-negative function $f_1$ such that either $f_1(x)=\sum_{y\ne x}f_1(y)$ or the support of $f_1$ is strictly contained in $A$. In the first case, we can finish by using the result for $\a=1/2$ and in the second case we can finish by using induction on the size of $A$.

Now note that if $D'$ has support size 2 and maximum $\a$, then $|\E_{x\sim D'}\omega_p^x|\leq|\a+(1-\a)\omega_p|$, and also that $|\a+(1-\a)\omega_p|$ decreases on $[0,1/2]$ and increases on $[1/2,1]$. Therefore, by the triangle inequality, $|\E_{x\sim D}\omega_p^x|\leq|\beta+(1-\beta)\omega_p|$, where $\beta=\max\{\a,1/2\}$.

Finally, we note that 
\begin{align*}
|\a+(1-\a)\omega_p|^2&=\a^2+2\a(1-\a)\cos(2\pi/p)+(1-\a)^2\\
&=1-2\a(1-\a)(1-\cos(2\pi/p)).\\
\end{align*} 
Since $1-\cos\theta\geq\theta^2/2-\theta^4/24$, which is at least $\theta^2/4$ when $\theta\leq 2\pi/3$, we obtain an upper bound of $1-2\beta(1-\beta)\pi^2/p^2$. If $\beta>1/2$, this is at most $1-(1-\a)\pi^2/p^2$, which is at most $1-c\pi^2/p^2$. If $\beta=1/2$, then the upper bound is $1-\pi^2/2p^2$, which again is at most $1-c\pi^2/p^2$ since $c\leq 1/2$.
\end{proof}

Recall that we are assuming that $p\geq 3$, which is why we could assume that $\theta\leq 2\pi/3$. However, with small changes the above lemma is clearly true for $p=2$ as well -- in this case $|\a+(1-\a)\omega_p|=|2\a-1|$, and we end up with a bound of $1-2c$.
\medskip

Given the lemma, the proof is very simple. 

\begin{proposition} \label{Equidistribution of high support linear forms} Let $p$ be a prime and let $c \in (0,1/2]$ be a positive real number. If $D$ is a probability distribution on $\F_p$ such that $D(x) \le 1-c$ for each $x \in \F_p$, and $l: \F_p^n \rightarrow \F_p$ is a linear form, then \[ |\bias_{t,D} l| \le (1 - c \pi^2/p^2)^{\rk_1(l)} \] for every $t \in \F_p^*$. \end{proposition}

\begin{proof} Writing $l:x \rightarrow \sum_{z=1}^n a_z x_z$, we write $\bias_{t,D}$ as a product  $\prod_{z=1}^n \E_{x_z \sim D} \omega_p^{t a_z l(x_z)}$. For each $z \in \lbrack n \rbrack$ the inner expectation is equal to $1$ if $a_z = 0$, and its modulus is otherwise always at most $1-c\pi^2/p^2$, by Lemma \ref{average of roots of unity}. 

The result follows. \end{proof}

Again, the result also holds for $p=2$, but with a bound of $(1-2c)^{\rk_1(l)}$.

\subsection{The bilinear case}

We next consider the case of a bilinear form $b: \F_p^{n_1} \times \F_p^{n_2} \rightarrow \F_p$.

\begin{proposition}\label{Equidistribution of high rank bilinear forms}
Let $p$ be a prime, let $c_1,c_2\in(0,1/2]$, and let $D_1,D_2$ be distributions on $\F_p$ such that $D_i(x)\leq 1-c_i$ for all $x\in\F_p$ and for $i=1,2$. Let $n_1,n_2$ be positive integers and let $b:\F_p^{n_1}\times\F_p^{n_2}\to\F_p$ be a bilinear form of rank $k$. Then $|\bias_{t,(D_1,D_2)}b|\leq 3e^{-c_1c_2\pi^2k/2p^2\log(2ep/c_1)}$ for every $t\in\F_p$.
\end{proposition}

\begin{proof}
Since $b$ has rank $k$ we can find subsets $X\subset[n_1]$ and $Y\subset[n_2]$ of size $k$ such that the restriction of $b$ to $\F_p^X\times\F_p^Y$ has rank $k$. For each $x\in\F_p^{n_1}$ let us write it as $(x_1,x_2)$, where $x_1\in\F_p^X$ and $x_2\in\F_p^{[n_1]\setminus X}$. We also let $L_x:\F_p^{Y}\to\F_p$ be the linear form $y\mapsto b(x,y)$ and let $a_x\in\F_p^{Y}$ be the vector such that $L_x(y)=\sum_{j\in[n_2]}(a_x)_jy_j$ for every $y\in\F_p^{Y}$. 

Let us now fix $x_2\in\F_p^{[n_1]\setminus X}$. Then the map $x_1\mapsto a_{(x_1,x_2)}$ is an affine map from $\F_p^X$ to $\F_p^Y$, with linear part $a_{(x_1,0)}$ of full rank. Therefore, it is a bijection. It follows that for every $a\in\F_p^Y$, the probability that $a_{(x_1,x_2)}=a$ when $x_1\sim D_1^X$ is at most $(1-c_1)^k$. In particular, for any given $u\leq k/3$, the probability that $a_{(x_1,x_2)}$ has support size at most $u$ is at most
\[(1-c_1)^k\sum_{j=0}^u\binom kj(p-1)^j\leq 2(1-c_1)^k(p-1)^u\binom ku\leq 2e^{-c_1k}(epk/u)^u.\]

If the support size is greater than $u$, then by Proposition \ref{Equidistribution of high support linear forms}, \[|\E_{y\sim D_2^{n_2}}\omega_p^{b(x,y)}|\leq (1-c_2\pi^2/p^2)^u\leq e^{-c_2\pi^2u/p^2}.\]
Therefore, for any $t\in\F_p$, any $x_2\in\F_p^{[n_1]\setminus X}$ and any $u\leq k$ we have the upper bound
\[|\E_{x_1\sim D_1^X, y\sim D^{n_2}}\omega_p^{tb((x_1,x_2),y)}|\leq 2e^{-c_1k}(epk/u)^u+e^{-c_2\pi^2u/p^2}.\]
Let us write $\theta$ for $u/k$. Applying the triangle inequality, we deduce that
\[|\bias_{t,(D_1,D_2)}b|\leq 2e^{-c_1k}(ep/\theta)^{\theta k}+e^{-c_2\pi^2\theta k/p^2}.\]
Setting $\theta=c_1/2\log(2ep/c_1)$, one can check that $(ep/\theta)^{\theta k}\leq e^{c_1k/2}$, and also that if $k\geq 1$ then $e^{-c_1k/2}\leq e^{-c_2\pi^2\theta k/p^2}$. We therefore obtain an upper bound of $3e^{-c_1c_2\pi^2k/2p^2\log(2ep/c_1)}$. If $k=0$ then this upper bound holds trivially.
\end{proof}

In this section we are proving our main results by showing that high rank implies low bias. Later in the paper, it will be more convenient to switch things round and prove that high bias implies low rank. With that in mind, we note the following simple corollary of Proposition \ref{Equidistribution of high rank bilinear forms}.

\begin{corollary} \label{flipped bilinear case} Let $p$ be a prime, let $c_1,c_2\in(0,1/2]$, and let $D_1,D_2$ be distributions on $\F_p$ such that $D_i(x)\leq 1-c_i$ for all $x\in\F_p$ and for $i=1,2$. Let $\e>0$, let $n_1,n_2$ be positive integers, and let $b:\F_p^{n_1}\times\F_p^{n_2}\to\F_p$ be a bilinear form with $|\bias_{t,(D_1,D_2)}b|\geq\e$ for some non-zero $t\in\F_p$. Then $b$ has rank at most $K_{p,c_1,c_2}(\epsilon)$, where \[K_{p,c_1,c_2}(\epsilon) = 2c_1^{-1}c_2^{-1}\pi^{-2}p^2\log(2ep/c_1)\log(3/\e).\]
\end{corollary}

\begin{proof}
This is just a back-of-envelope calculation, but for the convenience of the reader we include it. By Proposition \ref{Equidistribution of high rank bilinear forms}, we have the inequality
\[3e^{-c_1c_2\pi^2\rk b/2p^2\log(2ep/c_1)}\geq\e.\]
It follows that
\[c_1c_2\pi^2\rk b/2p^2\log(2ep/c_1)\leq\log(3/\e),\]
and from that we obtain the bound stated.
\end{proof}

\subsection{The quadratic case}

We now begin the proof of Theorem \ref{Main theorem} in the case of a polynomial of degree $2$. First, we recall the definition of box norms and a standard fact about them, which for convenience we give in full, since it may be hard to find a proof of the precise formulation we give. 

\begin{definition} Let $W_1$ and $W_2$ be finite sets and let $f:W_1\times W_2\to\C$. Let $U_1$ and $U_2$ be independent random variables with $U_i$ taking values in $W_i$ for $i=1,2$. Let $U_1'$ and $U_2'$ be copies of $U_1$ and $U_2$, respectively, with $U_1,U_1',U_2$ and $U_2'$ all independent. The \emph{box norm of $f$ with respect to $U_1$ and $U_2$} is the quantity $\|f\|_\square$ defined by the formula
\[\|f\|_\square^4=\E f(U_1,U_2)\overline{f(U_1,U_2')}\,\overline{f(U_1',U_2)}f(U_1',U_2').\]
\end{definition}

The standard fact we shall use is the following.

\begin{lemma}\label{Decoupling for one coordinate} Let $f,U_1,U_1',U_2,U_2'$ be as above. Then $\|f\|_\square\geq|\E f|$.
\end{lemma}

\begin{proof}
By Cauchy-Schwarz,
\begin{align*}
|\E_{U_1,U_2} f(U_1,U_2)|^2&\leq\E_{U_1}|\E_{U_2}f(U_1,U_2)|^2\\
&=\E_{U_1}\E_{U_2,U_2'}f(U_1,U_2)\overline{f(U_1,U_2')}.\\
\end{align*}
Squaring both sides and applying Cauchy-Schwarz again, this time taking $U_2$ and $U_2'$ outside the modulus sign and keeping $U_1$ inside, we conclude that $|\E f|^4\leq\|f\|_\square^4$.
\end{proof}

We will also use Theorem \ref{Disjoint rank minors theorem} for $d=2$, for which a linear bound is available. 

\begin{proposition} [\cite{K.}, Proposition 5.1]\label{Disjoint rank for matrices} Let $\F$ be a field, let $ A$ be an $ n \times n$ matrix over $ \F$, and let $k$ be a positive integer. If for each $n \times n$ diagonal matrix $D$ we have $\rk(A+D) \ge k$ then there exist disjoint $X,Y \subset \lbrack n \rbrack$ such that $\rk A(X \times Y) \ge k/3$. \end{proposition}

Our third preparatory result provides us with a useful sufficient condition for a function to have small bias.

We now introduce an analogue for quadratic forms of the notion of essential rank from Definition \ref{Essential and partition disjoint rank definition} for bilinear forms.

\begin{definition} Let $q: \F_p^n \rightarrow \F_p$ be a quadratic form. The \emph{essential rank} of the quadratic form is the quantity $\erk q= \min \rk (q+q_{\Delta})$, where $\rk$ is the notion of rank introduced in Definition \ref{rank of a polynomial}, and where the minimum is taken over all diagonal quadratic forms $q_\Delta:\F_p^n \rightarrow \F_p$, that is, forms of the type $x \mapsto \sum_{i=1}^n a_i x_i^2$ for some coefficients $a_1, \dots, a_n \in \F_p$. \end{definition}

\begin{lemma}\label{The Essential rank of a quadratic form is at least the essential rank of the associated bilinear form} Let $q: \F_p^n \rightarrow \F_p$ be a quadratic form and let $b: \F_p^n \times \F_p^n \rightarrow \F_p$ be a bilinear form such that $q(x) = b(x,x)$ for every $x \in \F_p^n$. Then $\erk q \le \erk b$. \end{lemma}

\begin{proof}
If $\erk b\leq k$, then we can find a diagonal form $b_\Delta$ such that $\rk(b+b_\Delta)\leq k$. Let $q_\Delta(x)=b_\Delta(x,x)$ for each $x$. Then $q_\Delta$ is a diagonal quadratic form, and $q(x)+q_\Delta(x)=b(x,x)+b_\Delta(x,x)$ for every $x$. Since $b+b_\Delta$ has rank at most $k$, we can write $b(x,y)+b_\Delta(x,y)$ in the form $\sum_{i=1}^kf_i(x)g_i(y)$. But then $q(x)+q_\Delta(x)=\sum_{i=1}^kf_i(x)g_i(x)$, so by the definition of rank we are using, $q+q_\Delta$ has rank at most $k$.
\end{proof}

The next result is a simple technical lemma, which we state separately from the proof of the main result of the section because it will be used twice more in the paper.

\begin{lemma}\label{Reduction to one function}
Let $p$ be a prime, let $D$ be a distribution on $\F_p$, let $f_0,f_1, \dots, f_k: \F_p^n \rightarrow \F_p$ and let $F: \F_p^k \rightarrow \F_p$. If 
\[|\bias_{t,D} (f_0 + a_1 f_1 + \dots + a_k f_k)| \le p^{-k} \epsilon \] 
for each $t \in \F_p^*$ and for all $(a_1, \dots, a_k) \in \F_p^{k}$, then 
\[|\bias_{t,D} (f_0 + F \circ (f_1, \dots, f_k))| \le \epsilon \] 
for each $t \in \F_p^*$. \end{lemma}

\begin{proof}
By definition,
\[\bias_{t,D}(f_0+F\circ(f_1,\dots,f_k))=\E_{x\sim D^n}\omega_p^{tf_0(x)}\omega_p^{tF(f_1(x),\dots,f_k(x))}.\]
Let $\phi(u_1,\dots,u_k)=\omega_p^{tF(u_1,\dots,u_k)}$ for each $u_1,\dots,u_k$. Then by the Fourier inversion formula,
\[\phi(u_1,\dots,u_k)=\sum_{a_1,\dots,a_k}\hat\phi(a_1,\dots,a_k)\omega_p^{a_1u_1+\dots+a_ku_k}.\]
Noting that $|\hat\phi(a_1,\dots,a_k)|\leq\|\phi\|_\infty\leq 1$ for every $a_1,\dots,a_k$, we may conclude that
\begin{align*}|\E_{x\sim D^n}\omega_p^{tf_0(x)}\omega_p^{tF(f_1(x),\dots,f_k(x))}|&\leq\sum_{a_1,\dots,a_k}|\E_{x\sim D^n}\omega_p^{tf_0(x)+a_1f_1(x)+\dots+a_kf_k(x)})|\\
&=\sum_{a_1,\dots,a_k}|\bias_{t,D}(f_0+t^{-1}(a_1f_1+\dots+a_kf_k))|.\\
\end{align*}
By our assumption, each summand in the last expression is at most $p^{-k}\e$, so the result follows.\end{proof}

We are now ready to start the proof of Theorem \ref{Main theorem} for polynomials of degree $2$.

\begin{proposition} Let $p \ge 3$ be a prime, let $S$ be a subset of $\F_p$ with $|S|\geq 2$, let $P: \F_p^n \rightarrow \F_p$ be a polynomial of degree $2$. If $|S|\geq 3$, then
\[|\bias_{t,S}P|\leq 3e^{-\rk_2P/2^{8}p^4(\log p)^2}\]
for every $t\ne 0$, while if $|S|=2$, then
\[|\bias_{t,S}P|\leq 3e^{-(\rk_SP-2)/64p^2\log p}.\]
\end{proposition}

\begin{proof} Let $t\in\F_p^*$ be fixed throughout the proof. We write $P = q + l + c$, where $q$, $l$, $c$ are the quadratic, linear and constant parts of $P$, respectively. Let $b: \F_p^n \times \F_p^n \rightarrow \F_p$ be the unique symmetric bilinear form such that $q(x) = b(x,x)$ for every $x \in \F_p^n$. Note that we have the polarization-type identity
\[q(x+y)-q(x+y')-q(x'+y)+q(x'+y')=2b(x-x',y-y'),\]
from which it follows that the same identity holds with $P$ replacing $q$.

Let $r=\erk q$. Then by Lemma \ref{The Essential rank of a quadratic form is at least the essential rank of the associated bilinear form} we have $\erk b \ge r$. By Proposition \ref{Disjoint rank for matrices} there exists a bipartition $\{X,Y\}$ of $\lbrack n \rbrack$ such that the restriction $b(\F^X \times \F^Y)$ has rank at least $r/3$. 

For $x$ chosen uniformly at random from $S^n$ let $U_1,U_1'$ be copies of $x(X)$ and let $U_2,U_2'$ be copies of $x(Y)$, with all four random variables being independent. Let $f(U_1,U_2)=\omega_p^{tP(U_1,U_2)}$, where we are writing $(U_1,U_2)$ for the vector $x\in\F_p^n$ such that $x(X)=U_1$ and $x(Y)=U_2$. We then have that $\bias_{t,S}P=\E f$. By Lemma \ref{Decoupling for one coordinate}, we also know that $|\E f|\leq\|f\|_\square$. But
\[\|f\|_\square^4=\E \omega_p^{t(P(U_1,U_2)-P(U_1,U_2')-P(U_1',U_2)+P(U_1',U_2'))}=\E\omega_p^{2tb(U_1-U_1',U_2-U_2')},\]
where the last equality comes from the polarization identity mentioned earlier.

Let $\mu_S$ be the uniform distribution on $S$ and let $D=\mu_S*\mu_{-S}$. That is, $D$ is the distribution on $\F_p$ where $D(u)=\P[v-w=u]$ when $v,w$ are chosen uniformly from $S$. Note that since $|S|\geq 2$, the probability $\P[v-w=u]$ is at most 1/2. 

The last expression above can be rewritten as 
\[\E_{x\sim D^X, y\sim D^Y}\omega_p^{2tb(x,y)}=\bias_{2t,(D,D)}b.\]
Since $\max_uD(u)\leq 1/2$ we can apply Proposition \ref{Equidistribution of high rank bilinear forms} with $c_1=c_2=1/2$ and $k=r/3$ to deduce that $|\bias_{2t,(D,D)}b|\leq 3e^{-\pi^2r/24p^2\log(4ep)}$. From this it follows that 
\begin{equation}\label{biasestimate}
|\bias_{t,S}P|\leq 3e^{-\pi^2r/96p^2\log(4ep)}.
\end{equation}

This is not yet what we want, because $r=\erk q$, which is not necessarily the same as $\rk_2P$. To complete the proof, we observe first that we can write $P$ in the form
\[P(x)=d(x)+\sum_{j=1}^r\phi_j(x)\psi_j(x)+\theta(x),\]
where $d$ is a diagonal quadratic form, the $\phi_j$ and $\psi_j$ are linear forms, and $\theta$ is an affine form (all the forms being from $\F_p^n$ to $\F_p$). 

By Lemma \ref{Reduction to one function},  $|\bias_{t,S}P|$ is at most $p^{2r+1}$ times the maximum over all $a_1,\dots,a_{2r+1}$ of
\[|\bias_{t,S}(d+a_1\phi_1+a_2\psi_1+\dots+a_{2r-1}\phi_r+a_{2r}\psi_r+a_{2r+1}\theta)|,\]
which is equal to
\[|\E_{x\in S^n}\omega_p^{td(x)+a_1\phi_1(x)+a_2\psi_1(x)+\dots+a_{2r-1}\phi_r(x)+a_{2r\psi(x)}+a_{2r+1}\theta(x)}|.\]

The expectation over $x$ can be written as a product of the form $\prod_{i=1}^n\E_{x_i\in S}\omega_p^{P_i(x_i)}$, where each $P_i:\F_p\to\F_p$ is a polynomial of degree at most 2. We now consider the cases $|S|\geq 3$ and $|S|=2$ separately.

Write $\supp(d)$ for the support of $d$, meaning that if $d(x)=\sum_i\lambda_ix_i^2$, then $\supp(d)=\{i:\lambda_i\ne 0\}$. Then for each $i\in d$ the polynomial $P_i$ has degree exactly 2 and therefore takes each value at most twice. In particular, if $|S|\geq 3$, then we can apply Lemma \ref{average of roots of unity} with $c=1/3$, to deduce that $|\E_{x_i\in S}\omega_p^{P_i(x_i)}|\leq 1-\pi^2/3p^2$. Therefore, 
\[|\E_{x\in S^n}\omega_p^{td(x)+a_1\phi_1(x)+a_2\psi_1(x)+\dots+a_{2r-1}\phi_r(x)+a_{2r\psi(x)}+a_{2r+1}\theta(x)}|\leq(1-\pi^2/3p^2)^{|\supp(d)|}.\]
It follows that 
\[|\bias_{t,S}P|\leq p^{2r+1}(1-\pi^2/3p^2)^{|\supp(d)|}\leq e^{(2r+1)\log p-\pi^2|\supp(d)|/3p^2}.\]

But $\rk_2d\leq|\supp(d)|$, so $r+|\supp(d)|+1\geq\rk_2P$, so for every $r,s$ with $r+s+1\leq\rk_2P$ we obtain an upper bound for $|\bias_{t,S}P|$ of 
\[\max\{3e^{-\pi^2r/96p^2\log(4ep)},e^{(2r+1)\log p-\pi^2s/3p^2}\}.\]
Choosing $s$ to be $\rk_2P/2$ and $r$ to be $\pi^2\rk_2P/32p^2\log p$, we obtain an upper bound of $3e^{-\rk_2P/2^{8}p^4(\log p)^2}$ after a back-of-envelope calculation.

Now let us assume that $|S|=2$. Let $Q$ be a polynomial that agrees with $P$ on $S^n$ and is such that $\rk Q=\rk_S P$. Let $Q=q+l+c$, where $q, l$ and $c$ are the quadratic, linear and constant parts of $Q$. If $Q$ has degree $2$ (as opposed to degree $1$, which is also a possibility even if $P$ has degree $2$), then let $d$ be a diagonal quadratic form such that $\rk(q+d)=\erk q$. Then 
\[\rk_SP\leq\rk_S(P+d)+1=\rk_S(Q+d)+1\le\rk_S(q+d)+2\leq\rk(q+d)+2=\erk q+2.\]
Therefore, $\erk q\geq\rk_SP-2$, so from \eqref{biasestimate} it follows that $|\bias_{t,S}P|\leq 3e^{-\pi^2(\rk_SP-2)/96p^2\log(4ep)}$, which a simple calculation shows is at most $3e^{-(\rk_SP-2)/64p^2\log p}$.

If $Q$ has degree $1$, then $\rk_1Q=\rk_SP$, so by Proposition \ref{Equidistribution of high support linear forms} with $D$ the uniform distribution on $S$ (and hence with $c=1/2$) we have $|\bias_{t,S}P|\leq(1-\pi^2/2p^2)^{\rk_SP}$, which is smaller than the bound just proved when $Q$ has degree 2. This completes the proof.
\end{proof}

\bigskip
\bigskip

\section{The proof in the case of an alphabet with size 2} \label{Section: The proof in the case of an alphabet with size 2}

In this section we prove our main theorem in the case where $|S| = 2$. Our argument can be summarized as follows. We begin by proving in particular that if $S$ is a non-empty subset of $\F_p$ with $|S| \ge 2$ and $A$ is a subset of $S^n$ dense inside $S^n$, then $(p-1)A$ is dense in $\F_p^n$. Using a small strengthening of this fact, and using the fact that the bias of a $d$-linear form $m: (\F_p^n)^d \rightarrow \F_p$ is the average of the biases of the $(d-1)$-linear forms obtained by fixing the first vector given to $m$, we are then able to prove a slight strengthening of Proposition \ref{bounded analytic rank implies bounded partition rank for restricted alphabets}. It then suffices to show that the task of proving that a polynomial $P$ of degree $d$ with high bias has bounded rank can be reduced to this strengthening. To do so, we shall associate a $d$-linear form $m$ with the polynomial $P$, and the assumption that $|S| = 2$ will allow us to ensure (in a way that does not work for larger alphabets) that $m$ has high essential partition rank. Applying Theorem \ref{Disjoint rank minors theorem} and using the obvious generalization of Lemma \ref{Decoupling for one coordinate} to $d $ variables we then be able to conclude in rather short order.

For the rest of the paper we shall use the following notation. Let $G$ be a finite abelian group. Given a probability distribution $D$ on $G$ we write $-D$ for the probability distribution on $G$ defined by $(-D)(x) = D(-x)$ for all $x \in G$. If $D$ and $D'$ are two probability distributions on $G$, then we write $D+D'$ for the probability distribution on $G$ defined by $(D+D')(x) = \sum_{y,z \in G: y+z=x} D'(y) D''(z)$ for all $x \in G$ -- that is, for the convolution of $D$ and $D'$, or equivalently for the distribution of the sum of two independent random variables, one with distribution $D$ and the other with distribution $D'$. We write $D-D'$ for the distribution $D+(-D')$. Given a probability distribution $D$ on $G$ and a positive integer $C$, we shall write $CD$ for the probability distribution $D + \dots + D$ ($C$ times) on $G$. If $A$ is a subset of $G$ and $D$ is a probability distribution on $G$ then we define the \emph{density of A with respect to D} by $\sum_{x \in A} D(x)$. In particular, the density of $A$ with respect to the uniform probability distribution on $G$ is the density $|A|/|G|$ of $A$ inside $G$ in the usual sense.

\subsection{Lemmas on the density of sumsets}

Later in the proof we shall find ourselves in a situation where we have a $d$-linear form $m:(\F_p^n)^d\to\F_p$ and a dense set $A$ of $u\in\F_p^n$ such that the $(d-1)$-linear form $m_u:(\F_p^n)^{d-1}\to\F_p$ defined by $m_u(x_2,\dots,x_d)=m(u,x_2,\dots,x_d)$ has low partition rank. However, this density will be with respect to a distribution $D^n$, whereas we would prefer the uniform distribution. The results of this section show that we can achieve this at the cost of passing to a suitable sumset of $A$, which, by the subadditivity of partition rank, will not be a problem for us. (This trick of passing to a sumset and exploiting subadditivity is often used to obtain a more structured set than $A$, though we shall not need the extra structure here.) 

Let $p$ be a prime. We begin by proving a result which, when iterated $p-2$ times, shows that if $A$ is a dense subset of $\{0,1\}^n$ then $(p-1)A$ is a dense subset of $\F_p^n$. 

\begin{proposition} \label{Density of sums of two sets} Let $r \ge 1$, $n \ge 1$ be positive integers, let $A \subset \{0, 1\}^n$ with density $\alpha$ inside $\{0,1\}^n$ and let $B \subset \{0, \dots, r\}^n$ with density $\beta$ inside $\{0, \dots, r\}^n$. Then $A+B$ has density at least $\alpha \beta$ inside $\{0, \dots, r+1\}^n$. \end{proposition}

\begin{proof}

For $f: \{0, 1\}^n \rightarrow \lbrack 0, \infty)$, $g: \{0, \dots, r\}^n \rightarrow \lbrack 0, \infty)$, the \emph{max-convolution} $f \circ g: \{0, \dots, r+1\}^n \rightarrow \lbrack 0, \infty)$ is defined by \[ f \circ g (z) = \max_{x+y = z} f(x) g(y). \] It suffices to show that \begin{equation} \E_z (f \circ g)(z) \ge (\E_x f(x)) (\E_y g(y)) \label{max-convolution inequality} \end{equation} where the expectations are taken over $\{0, \dots, r+1 \}^n$, $\{0,1\}^n$, and $\{0, \dots, r\}^n$, respectively, since if $f = 1_{A}$ and $g = 1_{B}$, then $f \circ g = 1_{A+B}$. The more general result is, however, convenient as an inductive hypothesis.

If $f$ and $g$ are constant functions equal to $\alpha$ and to $\beta$ respectively, then $f \circ g$ is the constant function equal to $\alpha \beta$, so in this case \eqref{max-convolution inequality} holds. For each $i \in \lbrack n \rbrack$ let $T_i$ be the operator that averages over the $i$th coordinate. That is, if $f:\{0,1\}^n\to[0,\infty)$ and $g:\{0,1,\dots,r\}^n\to[0,\infty)$, then
\[ T_if(x) = \E_{u \in \{0, 1\}} f(x_1, \dots, x_{i-1},u, x_{i+1}, \dots, x_n) \] 
and 
\[ T_ig(y) = \E_{u \in \{0, \dots, r\}} f(y_1, \dots, x_{i-1},u, x_{i+1}, \dots, x_n).\]
For each $i \in \lbrack n \rbrack$ the right-hand side of \eqref{max-convolution inequality} remains unchanged after replacing $f$ and $g$ by $T_if$ and $T_i g$, so to prove the inequality it suffices to show that we always have the inequality
\begin{equation} \E_z (T_i f \circ T_i g)(z) \le \E_z (f \circ g)(z). \label{averaging decreases the ell1 norm of max-convolution} \end{equation} Applying inequality \eqref{averaging decreases the ell1 norm of max-convolution} successively for $i=1, \dots, n$ then proves inequality \eqref{max-convolution inequality}.

To prove inequality \eqref{averaging decreases the ell1 norm of max-convolution}, we begin by showing that it follows from the one-dimensional case. Without loss of generality we can assume $i=n$. For each $z \in \{0, \dots, r+1 \}^n$, let $z' = (z_1, \dots, z_{n-1})$. Then we can write 
\[(f \circ g)(z',z_n) = \max_{x' + y' = z'} \max_{x_n + y_n = z_n} f(x',x_n) g(y',y_n) \] 
and
\[(T_n f \circ T_n g)(z',z_n) = \max_{x' + y' = z'} \E_{x_n} f(x',x_n) \E_{y_n} g(y',y_n) \] 
for all $z \in \{0, \dots, r+1 \}^n$. For a fixed $z' \in \{0, \dots, r+1 \}^{n-1}$ let $(x',y') \in \{0,1\}^{n-1} \times \{0, \dots, r \}^{n-1}$ such that $x'+y'=z'$ and $\E_{x_n} f(x',x_n) \E_{y_n} g(y',y_n)$ is maximized, and therefore equal to $\E_{z_n}(T_n f \circ T_n g)(z',z_n)$ for every $z_n$. If inequality \eqref{averaging decreases the ell1 norm of max-convolution} holds when $n=1$, then 
\[\E_{x_n} f(x',x_n) \E_{y_n} g(y',y_n) \le \E_{z_n} \max_{x_n + y_n = z_n} f(x',x_n) g(y',y_n).\] 
For each $z_n$, the maximum on the right-hand side is at most $(f\circ g)(z',z_n)$, so we deduce that
\[\E_{z_n} (T_n f \circ T_n g)(z',z_n) \le \E_{z_n} (f \circ g)(z',z_n).\] 
Averaging over all $z' \in \{0, \dots, r+1 \}^{n-1}$ we obtain inequality \eqref{averaging decreases the ell1 norm of max-convolution}.

In the $n=1$ case, writing $a_j = f(j)$ for each $j \in \{0, 1\}$ and $b_j = g(j)$ for each $j \in \{0, \dots, r \}$, inequality \eqref{averaging decreases the ell1 norm of max-convolution} can be rewritten as \begin{equation} \frac{r+2}{2(r+1)}(a_0 + a_1) (b_0 + \dots+ b_r) \le a_0 b_0 + (a_1 b_0 \vee a_0 b_1) + \dots + (a_1 b_{r-1} \vee a_0 b_r) + a_1 b_r. \label{1D inequality} \end{equation} Without loss of generality we can assume $a_1 \le a_0$ and that $a_0 = 1$. We write $a_1 = \mu a_0$ for some $\mu \le 1$. We now begin a second reduction where we show that it suffices to prove the inequality \eqref{1D inequality} in the case where $(b_0, \dots, b_r)$ is a geometric progression with ratio $\mu$. Indeed, assume that there exists $j \in \{ 0, r-1 \}$ such that for all $j' \in \{ 0, j-1 \}$, $b_{j'+1} = \mu b_{j'}$ but $b_{j+1} \neq \mu b_{j}$. Inequality \eqref{1D inequality} then simplifies to \begin{equation} \frac{r+2}{2(r+1)}(1+\mu) (b_0 + \dots+ b_r) \le b_0 + \dots + b_j + (\mu b_j \vee b_{j+1}) + \dots + (\mu b_{r-1} \vee b_r) + \mu b_r. \label{1D inequality with mu} \end{equation}

If $b_{j+1} > \mu b_j$ then there exists $t>0$ such that after decreasing $b_{j+1}$ by $t$ and increasing each of the quantities $b_0, \dots, b_j$ in such a way that $b_0 + \dots + b_j$ is increased by $t$ and $(b_0 \dots, b_j)$ remains a geometric progression with ratio $\mu$ after the modification, we have $b_{j+1} = \mu b_j$. This modification leaves the left-hand side of \eqref{1D inequality with mu} unchanged. On the right-hand side the sum $b_0 + \dots + b_j$ increases by $t$, the term $\mu b_j \vee b_{j+1}$ decreases by $t$, the term $\mu b_{j+1} \vee b_{j+2}$ (or instead, if $j=r-1$, the term $\mu b_r$) cannot increase, and all other terms are left unchanged, so the right-hand side cannot increase.

If $b_{j+1} < \mu b_j$ then there exists $t>0$ such that after increasing $b_{j+1}$ by $t$ and decreasing each of the quantities $b_0, \dots, b_j$ in such a way that $b_0 + \dots + b_j$ is decreased by $t$ and $(b_0 \dots, b_j)$ remains a geometric progression with ratio $\mu$ after the modification, we have $b_{j+1} = \mu b_j$. The left-hand side of \eqref{1D inequality with mu} is again unchanged. On the right-hand side the sum $b_0 + \dots + b_j$ decreases by $t$, the term $\mu b_j \vee b_{j+1}$ remains unchanged, the term $\mu b_{j+1} \vee b_{j+2}$ (or instead, if $j=r-1$, the term $\mu b_r$) increases by at most $\mu t \le t$, and all other terms are left unchanged, so the right-hand side cannot increase.

Iterating this process at most $r$ times we obtain $b_0, \dots, b_r$ in geometric progression with ratio $\mu$. In this case, assuming without loss of generality that $b_0 = 1$, the inequality \eqref{1D inequality with mu} becomes \[\frac{r+2}{2(r+1)}(1+\mu) (1+\mu+ \dots + \mu^r) \le 1 + \mu + \dots + \mu^{r+1} \] which simplifies to \[2 (\mu + \dots + \mu^r) \le r (1+\mu^{r+1}).\] This inequality holds, since by the weighted AM-GM inequality, for each $j \in \lbrack r \rbrack$ we have $\mu^j \le \frac{(r+1-j) + j \mu^{r+1}}{r+1}$ and moreover $\sum_{j=1}^r \frac{j}{r+1} = r/2$. \end{proof}

We shall use the next two lemmas to obtain a connection between the density of a subset $B$ of $S^n$ inside $S^n$ for some subset $S$ of $\F_p$ and the density of $B$ with respect to $D^n$ for some distribution $D$ on $\F_p$. 

\begin{lemma} \label{Writing a distribution as a convolution with 0,1 for F_p} Let $p\geq 3$ be a prime, let $U$ be the distribution on $\F_p$ that assigns probability 1/2 to 0 and 1, and let $D$ be a probability distribution on $\F_p$ such that $|D(x) - p^{-1}| \le p^{-2}$ for each $x \in \F_p$. Then there exists a probability distribution $E$ on $\F_p$ such that $D = U + E$. \end{lemma}

\begin{proof} Given $x\in\F_p$, write $[x]$ for the residue of $x$ in $\{0,1,\dots,p-1\}$. Then the convolution of $U$ with the function $x\mapsto(-1)^{[x]}$ takes the value 1 at 0 and 0 everywhere else. By translating this example, we can show that every function that takes the value 1 in one place and 0 everywhere else is the convolution of $U$ with a function that takes values in $\{-1,1\}$.

It follows that for every function $f:\F_p\to\R$ there is a function $g$ with $\|g\|_\infty\leq\sum_x|f(x)|$ such that $f=g*U$. Apply this to the function $f(x)=D(x)-p^{-1}$. Using the main hypothesis (which is in fact stronger than we need), we obtain a function $g$ with $\|g\|_\infty\leq p^{-1}$ such that $g*U=D-p^{-1}$, and hence $(g+p^{-1})*U=D$. Since $\|g\|_\infty\leq p^{-1}$, the function $g+p^{-1}$ is a probability distribution.
\end{proof}

It will not be of use to us, but the lemma requires only that $p$ should be odd. For even $p$ it fails, since for every distribution $E$ the sum of $U+E$ over the even residues is equal to the sum over the odd residues.

In the proof of the next lemma the constant $C(p,c)$ is the constant introduced in Proposition \ref{Equidistribution of high support linear forms}, which has a similar proof. For $p$ a prime and for $c>0$ let $M(p,c)= 2 \log p (\log C(p,c))^{-1}$.

\begin{lemma}\label{Mixing of distributions in F_p}

Let $p\geq 3$ be a prime and let $c>0$. Let $D$ be a probability distribution on $\F_p$ and suppose that $D(x) \le 1-c$ for every $x \in \F_p$. Then for every $M \ge 2p^2\log p/c\pi^2$, the distribution $MD$ satisfies $|MD(x) - p^{-1}| \le p^{-2}$ for each $x \in \F_p$. \end{lemma}

\begin{proof}
By the convolution identity and the Fourier inversion formula, 
\[MD(x) = \mathbb E_r(\hat D(r))^m\omega^{-rx},\]
where for each $r\in\F_p^n$, we define $\hat D(r)$ to be $\sum_y D(y)\omega^{ry}$.

If $r=0$, then $\hat D(r)=1$, while otherwise, by Lemma \ref{average of roots of unity} it has absolute value at most $1-c\pi^2/p^2$. It follows that $|MD(x)-p^{-1}|\leq(1-c\pi^2/p^2)^M$, which is at most $e^{-c\pi^2M/p^2}$. The result follows. \end{proof}

\begin{proposition} \label{Dense sumsets for distributions on $F_p$} Let $p\geq 3$ be a prime, let $c>0$, let $M=2p^2\log p/c\pi^2$, and let $ D$ be a probability distribution on $ \F_p$ such that $|D(x)| \le 1-c$ for each $x \in \F_p$. Then if $A$ is a subset of $\F_p^n$ with density $\epsilon$ inside $\F_p^n$ with respect to $D^n$, the sumset $(p-1)MA$ has density at least $\epsilon^{(p-1)M}$ inside $ \F_p^n$ with respect to the uniform distribution on $\F_p^n$. \end{proposition}

\begin{proof}
Lemma \ref{Mixing of distributions in F_p} implies that $|(MD)(\{x\}) - p^{-1}| \le p^{-2}$ for every $x \in \F_p$. Applying Lemma \ref{Writing a distribution as a convolution with 0,1 for F_p} to $MD$ we obtain a probability distribution $E$ on $\F_p$ such that $MD = U + E$. Because $A$ has density at least $\epsilon$ with respect to $ D^n$, the set $MA$ has density at least $\e^M$ in $\F_p^n$ with respect to the distribution $(MD)^n$, since if $x_1,\dots,x_M$ are chosen independently according to the distribution $D$, the probability that $x_1+\dots+x_M\in MA$ is at least the probability that each $x_i$ belongs to $A$. 

Suppose now that we choose $y$ and $z$ independently at random from $\F_p^n$, according to the distributions $U^n$ and $E^n$, respectively. Then $y+z$ is distributed according to $(U+E)^n=(MD)^n$, so the probability that $y+z\in MA$ is at least $\e^M$. It follows that there exists $z\in\F_p^n$ such that the density with respect to $U^n$ of the set $\{y\in\F_p^n:y+z\in MA\}=MA-z$ is at least $\e^M$. In other words, letting $Y=(MA-z)\cap\{0,1\}^n$, we have that $Y$ has density at least $\e^M$ inside $\{0,1\}^n$. Applying Proposition \ref{Density of sums of two sets} $p-2$ times we obtain that $(p-1)Y$ has density at least $\epsilon^{(p-1)M}$ inside $\F_p^n$ with respect to the uniform distribution on $\F_p^n$. Since $(p-1)(Y+z)$ is contained in $(p-1)MA$, it follows that $ (p-1)MA$ also has density at least $\epsilon^{(p-1)M}$. \end{proof}

\subsection{Equidistribution of multilinear forms}

We are now ready to prove a result about the equidistribution of multilinear forms which in particular implies Proposition \ref{bounded analytic rank implies bounded partition rank for restricted alphabets}, which bounds the partition rank of a tensor in terms of its bias with respect to a restricted alphabet. For this result we shall not need the hypothesis that $|S|=2$, so it applies for general alphabets. Throughout the remainder of the paper we will restrict attention to $d$-linear forms from $(\F^n)^d$ to $\F$, because only these will be relevant to the proof of Theorem \ref{Main theorem}. However, all the results and proofs that we provide for multilinear forms $(\F^n)^d \rightarrow \F$ can be generalized easily to multilinear forms from $\F^{n_1} \times \dots \times \F^{n_d}$ to $\F$. 
\medskip

As in the previous subsection, if we are given a $d$-linear form $m: (\F_p^n)^d \rightarrow \F_p$ and an element $u \in \F_p^n$, we write $m_u$ for the $(d-1)$-linear form from $(\F_p^{n})^{d-1}$ to $\F_p$ defined by 
\[ m_u (x_2,\dots,x_d) = m(u,x_2,\dots,x_d). \] 
Our proof will appeal to Theorem \ref{bounded analytic rank implies bounded partition rank}, the result we quoted earlier that bounds partition rank in terms of analytic rank when the alphabet is unrestricted. We continue to write $A_{d,\F}$ for the best function such that $\pr T\leq A_{d,\F}(\ar T)$ for every degree-$d$ tensor $T$ over the field $\F$. For a fixed prime $p$ and a fixed $c>0$, we define a family of functions $B_{d,p,c}: (0, \infty) \rightarrow (0, \infty)$ for all $d \ge 2$ by $B_{2,p,c} = K_{p,c,c}$ and for all $d \ge 3$, 
\[B_{d,p,c}(\epsilon) = A_{d,\F_p}((p-1)M(p,c)B_{d-1,p,c}(\epsilon/2) + (p-1)M(p,c) \log_p((\epsilon/2)^{-1})),\]
where $M(p,c)=2p^2\log p/c\pi^2$ is the constant $M$ from Proposition \ref{Dense sumsets for distributions on $F_p$}. Note that $B_{2,p,c}(\e)$ is the bound arising from Corollary \ref{flipped bilinear case} in the case $c_1=c_2=c$: that is, if two distributions $D_1,D_2$ on $\F_p$ both take maximum values at most $1-c$ and $b:\F_p^n\times\F_p^n\to\F_p$ is a bilinear form with bias at least $\e$, then $\rk b\leq B_{2,p,c}(\e)$. 

\begin{proposition}\label{Equidistribution of multilinear forms for distributions} Let $p\geq 3$ be a prime, let $d \ge 2$ be a positive integer, let $0<c\leq 1/2$ and let $D_1, \dots, D_d$ be distributions on $\F_p$ such that for each $1\leq i\leq d$ and each $x \in \F_p$ we have $D_{i}(x) \le 1-c$. Let $\epsilon > 0$. Let $m: (\F_p^n)^d \rightarrow \F_p$ be a $d$-linear form such that there exists $t\ne 0$ for which $|\bias_{t,(D_1, \dots, D_d)} m| \ge \epsilon$. Then the partition rank of $m$ is at most $B_{d,p,c}(\epsilon)$. \end{proposition}

\begin{proof} We proceed by induction on $ d$. The result holds for $ d=2$ by our choice of $B_{2,p,c}$ (as discussed just above). Now let $ d \ge 3$ and assume that the result holds for $d-1$. For each $ u \in \F_p^n$ such that $ \pr m_u \ge B_{d-1,p,c}(\epsilon/2)$, the inductive hypothesis guarantees that
\[ |\bias_{t,(D_2, \dots, D_d)} (m_u)| \le \epsilon/2\] 
for every $t\ne 0$. It follows that the set of such $u \in \F_p^n$ has density at most $1-\epsilon/2$ with respect to $D_1^n$, since otherwise we would have 
\[ |\bias_{t,(D_1, \dots, D_d)} m| < \epsilon/2+\e/2=\e, \] 
contradicting our assumption. By the subadditivity of the partition rank and Proposition \ref{Dense sumsets for distributions on $F_p$} the set of $u \in \F_p^n$ such that $ \pr m_u \le (p-1)M(p,c)B_{d-1,p,c}(\epsilon/2)$ has density at least $(\epsilon/2)^{(p-1)M(p,c)}$ with respect to the uniform distribution on $\F_p^n$.

A theorem of Lovett \cite[Theorem 1.7(i)]{Lovett}, proved independently by Kazhdan and Ziegler \cite[Lemma 2.2]{Kazhdan and Ziegler}, states that the analytic rank of a tensor is bounded above by its partition rank. Therefore,
\[\ar m_u \le (p-1)M(p,c)B_{d-1,p,c}(\epsilon/2)\] 
for all such $u \in \F_p^n$. Since
\[ \bias m = \E_{u \in \F_p^n} \bias (m_u) = \E_{u \in \F_p^n} p^{-\ar (m_u)}, \] 
we obtain the inequality
\[ \bias m \ge (\epsilon/2)^{(p-1)M(p,c)} p^{-(p-1)M(p,c)B_{d-1,p,c}(\epsilon/2)},\] 
and therefore 
\[\ar m \le (p-1)M(p,c)B_{d-1,p,c}(\epsilon/2) + (p-1)M(p,c) \log_p((\epsilon/2)^{-1}).\] 
It follows that $\pr m \le B_{d,p,c}(\epsilon)$, as desired. \end{proof}

Proposition \ref{bounded analytic rank implies bounded partition rank for restricted alphabets} follows as a special case of Proposition \ref{Equidistribution of multilinear forms for distributions} by taking $A_{d,p,(S_1, \dots, S_d)}: \lbrack 0, \infty) \rightarrow \lbrack 0, \infty)$ to be defined by \[A_{d,p,(S_1, \dots, S_d)}(r) = B_{d,p,1/2}(p^{-r})\] for all $r \ge 0$.

\subsection{Equidistribution of polynomials}

We now turn to polynomials. Our strategy is broadly the same as it was for the quadratic case proved in the last section: we obtain a $d$-linear form $m$ from $P$ and disjoint sets $X_1,\dots,X_d$ such that when we restrict $m$ to $\F_p^{X_1}\times\dots\times\F_p^{X_d}$, the rank of the restriction tends to infinity with the rank of $m$ itself. We then apply the result for multilinear forms to $m$ and deduce from it the corresponding result for $P$. However, at one point we shall make critical use of the assumption that $|S|=2$.

Let $d \ge 2$ be a positive integer and let $\e$ be an element of $\{-1,1\}^d$. We shall write $N(\e)$ for the number of indices $1\leq i\leq d$ such that $\e_i = -1$. For each $k \ge 1$, we say that a monomial $a \prod_{u=1}^n x_u^{s_u}$ with $a \in \F_p^*$ involves at least (resp. at most) $k$ pairwise distinct variables if the set $\{u \in \lbrack n \rbrack: s_u \ge 1\}$ has size at least (resp. at most) $k$. The next proposition is a generalization of Lemma \ref{Decoupling for one coordinate} to $d$ variables, except that we state it only for functions of the form $\omega_p^{f(U_1,\dots,U_d)}$, as these are the functions that concern us. Since the result is standard, we give only a sketch of the proof.

\begin{proposition} \label{Decoupling for an arbitrary function} Let $p$ be a prime, let $d \ge 2$ be a positive integer, let $W_1, \dots, W_d$ be finite sets, let $ U_1,\dots,U_d$ be jointly independent random variables taking values in $W_1, \dots, W_d$ respectively and let $ f: W_1 \times \dots \times W_d \rightarrow \F_p$ be a function. Then 
\begin{equation} |\E \omega_p^{f(U_1,\dots,U_d)})|^{2^d} \le \E \omega_p^{\sum_{\e \in \{-1,1\}^d} (-1)^{N(\e)} f(U_{1,\e_1},\dots,U_{d,\e_d})}\label{bound using decoupling} \end{equation}
where $ U_{i, -1}$ and $ U_{i,1}$ have the same distribution as $ U_i$ for each $ i \in \lbrack d \rbrack$, and the $ 2d$ variables $ U_{i,-1}, U_{i,1}$ with $ i \in \lbrack d \rbrack$ are jointly independent. \end{proposition}

\begin{proof}[Proof sketch]
The proof is basically the same as that of Lemma \ref{Decoupling for one coordinate} except that the Cauchy-Schwarz inequality is now applied to each of the $d$ variables instead of just to two variables. \end{proof}

We now establish the connection that we shall use to deduce the approximate equidistribution of a polynomial from that of a suitable associated multilinear form. Given a polynomial $P: \F_p^n \rightarrow \F_p$ and a partition $\{X_1,\dots,X_d\}$ of $\lbrack n \rbrack$, we write $ P_{\{X_1,\dots,X_d\}}$ for the polynomial obtained from $P$ by keeping only the monomials $a \prod_{u=1}^n x_u^{s_u}$ with $a\ne 0$ such that for each $1\leq i\leq d$ there exists $ u \in X_{i}$ with $s_u \ge 1$. 

\begin{proposition} \label{Decoupling inequality for polynomials} Let $p\geq 3$ be a prime, let $d \ge 2$ be a positive integer, let $D$ be a distribution on $\F_p$, let $P: \F_p^n \rightarrow \F_p$ be a polynomial, and let $ \{X_1,\dots,X_d\}$ be a partition of $ \lbrack n \rbrack$. If $P$ has degree at most $d$, then \begin{equation} |\E_{x \sim D} \omega_p^{P(x)}|^{2^d} \le \E_{x \sim D - D} \omega_p^{P_{\{X_1,\dots,X_d\}}(x)}. \label{decoupling inequality with $D-D$} \end{equation}  \end{proposition}

\begin{proof} We apply Proposition \ref{Decoupling for an arbitrary function}. Consider first the special case where $P$ is a monomial $P(x)=\prod_{u=1}^nx_u^{s_u}$. For each $1\leq i\leq d$ let $U_i$ be the random variable $\prod_{u\in X_i}x_u^{s_u}$, where $x\sim D^n$, so that $P(x)=U_1\dots U_d$.

We now look at the behaviour of the quantity
\[\sum_{\e \in \{-1,1\}^d} (-1)^{N(\e)} f(U_{1,\e_1},\dots,U_{d,\e_d}).\]
If there exists $i$ such that $s_u=0$ for every $u\in X_i$, then $f(U_{1,\e_1},\dots,U_{d,\e_d})$ is independent of $\e_i$, from which it follows that the whole sum is zero. The only other possibility, since $P$ has degree at most $d$, is if for each $i$ there is exactly one $u_i$ such that $s_{u_i}=1$, and all other $s_u$ are zero. In that case for each $i$ let $U_{i,1}=x_{u_i}$ and $U_{i,-1}=x_{u_i}'$, where $x'$ is an independent copy of $x$. Then the quantity is equal to $\prod_{i=1}^d(U_{i,1}-U_{i,-1})=\prod_{i=1}^d(x_{u_i}-x_{u_i}')$.

By linearity it follows more generally that 
\[\sum_{\e \in \{-1,1\}^d} (-1)^{N(\e)} f(U_{1,\e_1},\dots,U_{d,\e_d})=P_{\{X_1,\dots,X_d\}}(x-x').\]
The result therefore follows from Proposition \ref{Decoupling for an arbitrary function}.
\end{proof}

Let $S$ be a subset of $\F_p$ of size $2$, which will remain fixed until the end of the section. Before we start the main proof we need a few more results about polynomials and their connections to multilinear forms. For the next three lemmas let $P:\F_p^n\to\F_p$ be a polynomial of degree $d$, and let $P=Q+R$ be the unique decomposition such that $Q$ is a linear combination of monomials of degree at most 1 in each variable separately and of total degree $d$, and $R$ is a linear combination of monomials such that either at least one variable has degree greater than 1 or the monomial has total degree less than $d$.

We shall make use of the following decomposition: for $p$ a prime, for $d \ge 2$ a positive integer, and for $P: \F_p^n \rightarrow \F_p$ a polynomial of degree $d$ over $\F_p$, we can write $P = Q+R$, where $Q$ is a linear combination of monomials $x_{i_1} \dots x_{i_d}$ with $i_1, \dots, i_d$ distinct elements of $\lbrack n \rbrack$, and $R$ is a linear combination of monomials $x_{i_1} \dots x_{i_{d'}}$ such that either $d'<d$, or else $d'=d$ and $i_1, \dots, i_d$ are not distinct.

The next lemma is where we shall use the assumption that $|S| = 2$: it is false in general for $|S| \ge 3$.
 
\begin{lemma}\label{Remainder with rank at most 1 for sets of size 2} The polynomial $R$ coincides on $S^n$ with a polynomial of degree at most $d-1$. \end{lemma}

\begin{proof} Each monomial in $R$ can be written as $x_{i_1}^{s_1} \dots x_{i_{d'}}^{s_{d'}}$ for some nonnegative integer $0 \le d' \le d-1$, some $i_1, \dots, i_{d'} \in \lbrack n \rbrack$, and some positive integers $s_1, \dots, s_{d'} \ge 1$. If all powers $s_1, \dots, s_{d'}$ are equal to $1$, then this monomial has degree at most $d-1$. If on the other hand one of the powers $s_1, \dots, s_{d'}$ is at least $2$, then
without loss of generality this power is $s_1$. Since $S$ has size $2$, there exist $a,b \in \F_p$ such that $x^{s_1} = ax+b$ for each of the two $x \in S$, we obtain that the monomial $x_{i_1}^{s_1} \dots x_{i_{d'}}^{s_{d'}}$ coincides on $S^n$ with the polynomial 
\[(a x_{i_1}+b)x_{i_1}^{s_2} \dots x_{i_{d'}}^{s_{d'}},\] 
which has degree at most $d-1$. \end{proof}

The next lemma is a polarization identity for $d$-linear forms. 

\begin{lemma} \label{unique symmetric multilinear form} There exists a unique symmetric $d$-linear form $m: (\F_p^n)^d \rightarrow \F_p$ such that \begin{equation} Q(y) = m(y, \dots, y) \label{polynomial equals symmetric multilinear form} \end{equation} for every $y \in \F_p^n$. \end{lemma}

\begin{proof} Let $Q$ be given by the formula 
\[Q(x) = \sum_{\{i_1<\dots<i_d\}} a_{\{i_1, \dots, i_d\}} x_{i_1} \dots x_{i_d}.\] 
Then we can define $m$ by the formula
\[m(y_1,\dots,y_d)=\sum_{j_1,\dots,j_d}b_{j_1,\dots,j_d}(y_1)_{j_1}\dots (y_d)_{j_d},\]
where $b_{j_1,\dots,j_d}=(d!)^{-1}a_{i_1,\dots,i_d}$ if $(j_1,\dots,j_d)$ is a permutation of $(i_1,\dots,i_d)$ and is zero if $j_1,\dots,j_d$ are not all distinct. Then it is not hard to check that $Q(y)=m(y,y,\dots,y)$ for every $y\in\F_p^n$. 

The uniqueness of $m$ follows from the fact that a non-zero polynomial of degree less than $p$ does not take the value zero everywhere, combined with the symmetry of $m$ and the fact that
\[a_{i_1,\dots,i_d}=\sum_{\sigma\in S_d}b_{i_{\sigma(1)},\dots,i_{\sigma(d)}}\]
for every $i_1,\dots,i_d$.
\end{proof}

Let $\mathcal V_d$ be the class of polynomials that are linear combinations of monomials $x_{i_1} \dots x_{i_d}$ for which $i_1, \dots, i_d$ are not all distinct. We define the \emph{essential rank} of a homogeneous polynomial $Q$ of degree $d$ to be $\min_{V\in\mathcal V_d} \rk (Q+V)$, and we denote it by $\erk Q$. Note that this agrees with our previous definition when $d=2$.

\begin{lemma}\label{Essential rank of the polynomial at least Essential rank of the symmetric multilinear form} Let $m$ be the symmetric multilinear form $m$ defined in Lemma \ref{unique symmetric multilinear form}. Then the essential partition rank of $m$ is at least the essential rank of $Q$. \end{lemma}

\begin{proof} Suppose that $\epr m \le k$. Then there exists a $d$-linear form $m': (\F_p^n)^d \rightarrow \F_p$ such that $ m'_{(i_1,\dots,i_d)} = 0$ whenever $ i_1,\dots,i_d \in \lbrack n \rbrack$ are distinct and such that $ \pr (m + m') \le k$. We evaluate \[ (m+m')(y, \dots, y) = m(y, \dots, y) + m'(y, \dots, y) \] for all $y \in \F_p^n$: the first term of the right-hand side is $ Q(y)$, and by the condition on the $ m'_{(i_1,\dots,i_d)}$ with $ i_1,\dots,i_d \in \lbrack n \rbrack$ pairwise distinct, the second term $ m'(y, \dots, y)$ is a linear combination of monomials of the type $y_{i_1}^{s_1}\dots y_{i_d}^{s_d}$ with $i_1,\dots,i_d$ not pairwise distinct, so we can write it as $ V(y)$ for some polynomial $ V$ spanned by these monomials, which shows that $ \erk Q \le \rk (Q+V)$.

Because $ \pr (m + m') \le k$ there exist for each $i \in \lbrack k \rbrack$ a bipartition $\{J_i, J_i'\}$ of $\lbrack d \rbrack$ (with $J_i, J_i'$ both non-empty) and multilinear forms $ M_{i,1}: (\F_p^n)^{J_i} \rightarrow \F_p$, $M_{i,2}: (\F_p^n)^{J_i'} \rightarrow \F_p$  such that 
\[ (m + m'): (y^1,\dots,y^d) \mapsto \sum_{i=1}^k M_{i,1}(y(J_i)) M_{i,2}(y(J_i')).\] 
For each $i \in \lbrack k \rbrack$, let the polynomials $Q_i, R_i$ be defined by $ Q_i(y) = M_{i,1}(y, \dots, y)$ and $ R_i(y) = M_{i,2}(y, \dots, y)$. Then 
 \[ (Q+V)(y) = \sum_{i=1}^k Q_i(y) R_i(y)\] 
 for every $y \in \F_p^n$. Because the sets $J_i,J_i'$, $i \in \lbrack k \rbrack$ are all strict subsets of $\lbrack d \rbrack$ we have $\deg Q_i, \deg R_i < \deg Q$ for each $i \in \lbrack k \rbrack$. Therefore, $ \erk Q \le k$. \end{proof}

We are now ready to prove the main result of this section. In the proof we shall make use of Theorem \ref{Disjoint rank minors theorem}. Recall that for each $d$ this theorem yields a function $\Lambda_d:\N\to\N$ such that if the essential partition rank of a tensor is at least $\Lambda_d(l)$, then the disjoint partition rank is at least $l$.

\begin{proposition}\label{Main theorem for alphabets of size 2} Let $p \ge 3$ be a prime, let $d \ge 2$ be a positive integer, let $S$ be a subset of $\F_p$ of size $2$, and let $\epsilon > 0$. Let $P: \F_p^n \rightarrow \F_p$ be a polynomial of degree $d$ and suppose that there exists $t \in \F_p^*$ with $|\bias_{t,S} P| \geq \epsilon$. Then $\rk_S P \le \Lambda_d(B_{d,p,1/2}(\epsilon^{2^d})) + 1$. \end{proposition}

\begin{proof} If $P$ agrees with a lower-degree polynomial $P_1$ on $S^n$, then we can replace $P$ by $P_1$ and apply induction on $d$. Otherwise, we note first that $\rk_S P \le (\erk Q) + 1$, since if $R_0$ is a linear combination of monomials of the type $y_{i_1}^{s_1}\dots y_{i_d}^{s_d}$ with $i_1,\dots,i_d$ not all distinct then by applying Lemma \ref{Remainder with rank at most 1 for sets of size 2} to $R+R_0$ we have $\rk_S (R+R_0) \le 1$, which implies that
\[\rk_SP\leq\rk(Q-R_0)+\rk_S(R+R_0)\leq\rk(Q-R_0)+1.\]
(The assumption that $P$ does not agree with a lower-degree polynomial on $S^n$ gives us the subadditivity here.) Choosing $R_0$ such that $\erk Q=\rk(Q-R_0)$ gives the bound claimed. 

If $\erk Q \ge \Lambda_d(B_{d,p,1/2}(\epsilon^{2^d}))$ then by Lemmas \ref{Remainder with rank at most 1 for sets of size 2} and \ref{Essential rank of the polynomial at least Essential rank of the symmetric multilinear form} there exists a symmetric $d$-linear form $m:(\F_p^n)^d \rightarrow \F_p$ such that $Q(y) = m(y, \dots, y)$ for every $y \in \F_p^n$, and such that $\epr m \ge \Lambda_d(B_{d,p,1/2}(\epsilon^{2^d}))$. By Theorem \ref{Disjoint rank minors theorem} we can find pairwise disjoint subsets $X_1, \dots, X_d \subset \lbrack n \rbrack$ such that $\pr m(\F^{X_1} \times \dots \times \F^{X_d}) \ge B_{d,p,1/2}(\epsilon^{2^d})$. We now apply Proposition \ref{Decoupling inequality for polynomials}. Letting $D$ be the distribution on $\F_p$ defined by $D(x) = |S|^{-2} \sum_{y,z \in S} S(y) S(z) 1_{y-z=x}$, we have \begin{equation} |\E_{x \in S^n} \omega_p^{tP(x)}|^{2^d} \le \E_{x \sim D} \omega_p^{tP_{\{X_1,\dots,X_d\}}(x)} \label{decoupling inequality used in the proof for size 2 alphabets} \end{equation} for each $t \in \F_p^*$. We can write \begin{align*} P_{\{X_1,\dots,X_d\}}(y) &= \sum_{\sigma \in \mathcal{S}_d} m(\F^{X_{\sigma(1)}} \times \dots \times \F^{X_{\sigma(D)}})(y(X_{\sigma(1)}),\dots,y(X_{\sigma(D)})) \\ &= d! m(\F^{X_1} \times \dots \times \F^{X_d})(y(X_1), \dots, y(X_d)) \end{align*} where the second equality follows from the symmetry of $m$. Since $D$ has maximum value at most $1/2$, and $\pr m(\F^{X_1} \times \dots \times \F^{X_d}) \ge B_{d,p,1/2}(\epsilon^{2^d})$, by Proposition \ref{Equidistribution of multilinear forms for distributions} and \eqref{decoupling inequality used in the proof for size 2 alphabets} we get $|\bias_{t,S} P| < \epsilon$ for all $t \in \F_p^*$, which is incompatible with our assumption. So $\erk Q \le \Lambda_d(B_{d,p,1/2}(\epsilon^{2^d}))$. The result follows. \end{proof}

\section{Equidistribution of combinations of multilinear forms with several choices of powers} \label{Section: Equidistribution of multi-alphabet multilinear forms}

In the previous section, where we proved Theorem \ref{Main theorem} in the case where the subset $S$ of $\F_p$ had size $2$, part of the proof of Proposition \ref{Main theorem for alphabets of size 2} relied heavily on Claim \ref{Remainder with rank at most 1 for sets of size 2}, which ensured that if $P$ had high rank, then its part $Q$ made up of monomials of the form $x_{i_1} \dots x_{i_d}$ necessarily had high essential rank. For $|S| \ge 3$ this fact becomes false, so we have to consider the case where $P$ has high rank but $Q$ has bounded essential rank. 

Let us very briefly sketch the argument that will follow. Given a polynomial $P:\F_p^n\to\F_p$ of degree $d$ and an alphabet $S$ of size at least $d+1$ (if it has size at most $d$ then we can replace $P$ by a polynomial of degree at most $d-1$ that takes the same values on $S^n$), we write $P$ as a linear combination of monomials, and then split it up according to the forms of those monomials -- that is, the sequence of indices used, in non-increasing order. Given a non-increasing sequence $s=(s_1,\dots,s_k)$ of positive integers with $s_1+\dots+s_k\leq d$, we write $P_s$ for the polynomial obtained when we retain just the linear combination of monomials of the form $x_{i_1}^{s_1}\dots x_{i_k}^{s_k}$, where $i_1,\dots,i_k$ are distinct numbers between 1 and $n$. We also write $|s|$ for the length of the sequence $s$, that is, for the number $k$.

With each polynomial $P_s$ with $|s|=k$ we can associate a polynomial $Q_s$ in $k$ variables $x(1),\dots,x(k)\in\F_p^n$ of the form $Q_s(x(1),\dots,x(k))=b_s(x(1)^{s_1},\dots,x(k)^{s_k})$, such that $b_s$ is a $k$-linear form, $P_s(x)=Q_s(x,x,\dots,x)$ for every $x$, and $b_s$ is symmetric in $x(i)$ and $x(j)$ whenever $s_i=s_j$. A key lemma, which we shall prove in this section, will be that if any one of the $|s|$-linear forms $b_s$ has high rank (the notion of essential rank does not arise here because we regard the variables $x(1),\dots,x(k)$ as belonging to distinct copies of $\F_p^n$), then $P$ has small bias on $S^n$. Therefore, if $P$ has large bias on $S^n$, we may conclude that the multilinear forms $b_s$ all have low rank. In the next section, with the help of Theorem \ref{Disjoint rank minors theorem}, we shall deduce from this that $P$ agrees with a low-rank polynomial on $S^n$. 
\bigskip

\iftrue
\else

In order to illustrate the main additional difficulty let us sketch the proof of Theorem \ref{Main theorem} in the case $\deg P = 3$ and $|S| \ge 4$.

Let us assume that $P$ has large enough rank and let us show that $|\bias_{t,S} P|$ is small for all $t \in \F_p^*$. If for some nonnegative integer $r$ we have $\erk Q \le r$, then we can write \[ Q = V + \sum_{i=1}^r Q_i R_i\] where $V$ is a polynomial where all monomials involve at most two distinct variables and $Q_i, R_i$ are polynomials with degree at most $2$. By Lemma \ref{Reduction to one function} it suffices to show that for every $a \in \F^r$ and $b \in \F^r$ the linear combination \begin{equation} P_{new}:= V+R+ \sum_{i=1}^r a_i Q_i + \sum_{i=1}^r b_i R_i \label{first reduction for polynomials of degree 3} \end{equation} satisfies that $|\bias_{t,S} P_{new}|$ is small for each $t \in \F_p^*$. Since the rank of $P$ is large enough and because the part $\sum_{i=1}^r a_i Q_i + \sum_{i=1}^r b_i R_i$ has degree at most $2$, we know that the linear combination $P_{new}$ still has degree $3$ and high rank. Furthermore all of its monomials involve at most two distinct variables, as all remaining terms are of the following three types: \begin{enumerate} \item $x_i^2x_j$ with $i,j$ distinct \item $x_ix_j$ with $i,j$ distinct \item$x_i^a$ for some $a \in \{0,1,2,3\}$. \end{enumerate} We decompose $P_{new}$ as the sum of three polynomials $(Q_{new})_{2,1}, (Q_{new})_{1,1}$ and $R_{new}$ corresponding to these three respective parts.

If $X,Y$ are disjoint subsets of $\lbrack n \rbrack$, then the contribution of $R_{new}$ to the polynomial $(P_{new})_{\{X,Y\}}$ is zero, but both $(Q_{new})_{2,1}, (Q_{new})_{1,1}$ contribute rather than only $(Q_{new})_{2,1}$. Letting $Ev_2:(y_1, \dots, y_n) \rightarrow (y_1^2, \dots, y_n^2)$ we can construct a unique bilinear form $b_{2,1}$ and a unique symmetric bilinear form $b_{1,1}$ such that \[ (Q_{new})_{2,1}(y) = b_{2,1}(Ev_2(y),y) \text{ and } (Q_{new})_{1,1}(y) = b_{1,1}(y,y) \] hold for all $y \in \F_p^n$, and such that $b_{2,1}$ (resp. $b_{1,1}$) has high rank provided that $(Q_{new})_{2,1}$, (resp. $(Q_{new})_{1,1}$) has high rank. We can then write \begin{equation} (P_{new})_{\{X,Y\}}(y) = b_{2,1}(Ev_2(y(X)), y(Y)) + b_{2,1}((y(X)), Ev_2(y(Y)))+ 2 b_{1,1}(y(X),y(Y)) \label{sum of two bilinear forms with evaluation} \end{equation} for all $y \in \F_p^n$.

One technical difficulty is to show that it suffices that $b_{2,1}(\F_p^X \times \F_p^Y)$ has high rank for the linear combination \eqref{sum of two bilinear forms with evaluation} to be have small bias, when $y$ is distributed according to $D^n$ for some probability distribution $D$ such that $\max_{x \in \F_p^*} D(x)$ is bounded away from $1$. In fact, we will be able to show that it suffices for either of $b_{2,1}(\F_p^X \times \F_p^Y)$, $b_{2,1}(\F_p^Y \times \F_p^X)$ or $b_{1,1}(\F_p^X \times \F_p^Y)$ to have high rank. The generalisation of this difficulty to higher-order linear combinations of multilinear forms will be the focus of the present section.

If neither $(Q_{new})_{2,1}$ nor $(Q_{new})_{1,1}$ has high rank, then we can perform another reduction similar to \eqref{first reduction for polynomials of degree 3}, this time reducing to a polynomial $P_{new,2}$ which is linear combination of monomials each involving at most one variable: writing \[P_{new,2}(x) = \sum_{i=1}^n P_i(x_i) \] for some polynomials $P_i: \F_p \rightarrow \F_p$ with degree at most $3$ we can factor \[|\E_{x \in S^n} \omega_p^{tP(x)}| = \prod_{i=1}^n |\E_{x_i \in S} \omega_p^{tP_i(x_i)}|.\] Since $|S| \ge 4$ the right-hand side is at most $C(p,|S|^{-1})^{|\{i \in \lbrack n \rbrack: P_i \neq 0\}|} \le C(p,|S|^{-1})^{(\rk P_{new,2}) - 1}$, which finishes the proof.
\fi

We now begin adapting some of the results that led to the proof of Proposition \ref{Equidistribution of multilinear forms for distributions}, which stated that multilinear forms that are significantly biased with respect to product distributions that are not too close to being atomic have low partition rank. Those results concerned a distribution $D$ on $\F_p$ that is not concentrated at a single point. We now need to generalize them to results concerning a distribution on $\F_p^k$ that is not concentrated on a proper affine subspace of $\F_p^k$.
\medskip

The first result we shall adapt is Lemma \ref{Writing a distribution as a convolution with 0,1 for F_p}. 

\begin{lemma}\label{Writing a distribution as a convolution with 0,1} Let $p \ge 3$ be an odd integer, let $k \ge 1$ be a positive integer, let $U$ be the uniform distribution on the subset $\{0,1\}\subset\F_p$, and let $D$ be a distribution on $\F_p^k$ such that for each $x \in \F_p^k$, $|D(x) - p^{-k}| \le p^{-2k}$. Then there exists a probability distribution $E$ on $\F_p^k$ such that $D = U^k + E$. \end{lemma}

\begin{proof} We saw in the proof of Lemma \ref{Writing a distribution as a convolution with 0,1 for F_p} that there is a $\pm 1$-valued function $\phi$ on $\F_p$ such that $U*\phi$ takes the value 1 at 0 and 0 everywhere else. It follows that $U^k*\phi^k$ takes the value 1 at 0 (where now $0$ is an element of $\F_p^k$) and 0 everywhere else, where by $\phi^k$ we mean the function $\phi^k(x)=\phi(x_1)\dots\phi(x_k)$.

Just as in the one-dimensional case, it follows that for every function $f:\F_p^k\to\R$ there is a function $g:\F_p^k\to\R$ with $\|g\|_\infty\leq\sum_x|f(x)|$ such that $f=g*U$. We apply this to the function $f(x)=D(x)-p^{-k}$, noting that $\sum_x|f(x)|\leq p^{-k}$ by our hypothesis. Then $g*U=D-p^{-k}$, from which it follows that $(g+p^{-k})*U=D$. Since $\|g\|_\infty\leq p^{-k}$, the function $g+p^{-k}$ is a probability distribution. \end{proof}

Proposition \ref{Density of sums of two sets} will be used as is, and we start by adapting Lemma \ref{Writing a distribution as a convolution with 0,1 for F_p}, Lemma \ref{Mixing of distributions in F_p}, and Proposition \ref{Dense sumsets for distributions on $F_p$}. For $k \ge 1$ a positive integer and $p \ge 3$ a prime, let $U_{\{0,1\},k,p}$ be the distribution on $\F_p^k$ defined by $U_{\{0,1\},k,p}(x) = 2^{-k}$ if $ x_1, \dots, x_k \in \{0,1\}$ and $U_{\{0,1\},k,p}(x) = 0$ otherwise.

Now we shall modify Lemma \ref{Mixing of distributions in F_p}. For $p$ a prime, for $k$ a positive integer, and for $c>0$, let $M(p,c,k)= 2k \log p / \log (C(p,c)^{-1}) = k M(p,c)$.

\begin{lemma}\label{Mixing of distributions}

Let $p$ be a prime, let $k \ge 1$ be a positive integer, and let $c>0$. If $D$ is a distribution on $\F_p^k$ such that $D(W) \le 1-c$ for every strict affine subspace $W$ of $\F_p^k$ then for all $M \ge 2kp^2\log p/c\pi^2$ the distribution $MD$ satisfies $|MD(x) - p^{-k}| \le p^{-2k}$ for each $x \in \F_p^k$. \end{lemma}

\begin{proof}
The proof is essentially the same as that of Lemma \ref{Mixing of distributions in F_p}. The one thing we need to observe is that the condition on affine subspaces implies an upper bound on the size of each non-trivial Fourier coefficient of $D$. Indeed,
\[\hat D(r)=\sum_yD(y)\omega^{r.y}=\sum_t\Big(\sum_{r.y=t}D(y)\Big)\omega^t\leq 1-c\pi^2/p^2,\]
where the inequality follows from Lemma \ref{average of roots of unity} and the fact that $\sum_{r.y=t}D(y)\leq 1-c$ for each $r,t$, by hypothesis. 

This time, the probability that $r=0$ is $p^{-k}$, so we deduce that $|MD(x)-p^{-k}|\leq(1-c\pi^2/p^2)^M\leq e^{-c\pi^2M/p^2}$. The result follows from our assumed lower bound on $M$. 
\end{proof}

We now generalize Proposition \ref{Dense sumsets for distributions on $F_p$}. Again, the generalization is straightforward, but we write it out in full, just to be clear about the details of the small changes needed.

\begin{proposition} \label{Dense sumsets for distributions on $F_p^k$} Let $p\geq 3$ be a prime, let $k \ge 1$ be a positive integer, let $c>0$, let $M=2kp^2\log p/c\pi^2$, and let $D$ be a probability distribution on $ \F_p^k$ such that $ D(W) \le 1-c$ for every strict affine subspace $ W$ of $ \F_p^k$. Then if $A$ is a subset of $ (\F_p^k)^n$ with density $\epsilon$ inside $ (\F_p^k)^n$ with respect to the distribution $D^n$, then $ (p-1)MA$ has density at least $\epsilon^{(p-1)M}$ inside $(\F_p^k)^n$ with respect to the uniform distribution on $(\F_p^k)^n$. \end{proposition}

\begin{proof} Lemma \ref{Mixing of distributions} implies that $|(MD)(\{x\}) - p^{-k}| \le p^{-2k}$ for every $x \in \F_p$. Applying Lemma \ref{Writing a distribution as a convolution with 0,1} to $MD$ we obtain a probability distribution $E$ on $\F_p$ such that $MD = U^k + E$. Because $A$ has density at least $\epsilon$ with respect to $D^n$, the set $MA$ has density at least $\e^M$ in $(\F_p^k)^n$ with respect to the distribution $(MD)^n$, since, as before, if $x_1,\dots,x_M$ are chosen independently according to the distribution $D$, the probability that $x_1+\dots+x_M\in MA$ is at least the probability that each $x_i$ belongs to $A$. 

Suppose now that we choose $y$ and $z$ independently at random from $(\F_p^k)^n$, according to the distributions $(U^k)^n$ and $E^n$, respectively. Then $y+z$ is distributed according to $(U^k+E)^n=(MD)^n$, so the probability that $y+z\in MA$ is at least $\e^M$. It follows that there exists $z\in(\F_p^k)^n$ such that the density with respect to $(U^k)^n$ of the set $\{y\in(\F_p^k)^n:y+z\in MA\}=MA-z$ is at least $\e^M$. In other words, letting $Y=(MA-z)\cap\{0,1\}^n$, we have that $Y$ has density at least $\e^M$ inside $(\{0,1\}^k)^n$. Identifying this with $\{0,1\}^{kn}$ and $(\F_p^k)^n$ with $\F_p^{kn}$, we can apply Proposition \ref{Density of sums of two sets} $p-2$ times to obtain the conclusion that $(p-1)Y$ has density at least $\epsilon^{(p-1)M}$ inside $(\F_p^k)^n$ with respect to the uniform distribution on $(\F_p^k)^n$. Since $(p-1)(Y+z)$ is contained in $(p-1)MA$, it follows that $ (p-1)MA$ also has density at least $\epsilon^{(p-1)M}$. 
\end{proof}

In what follows, we shall often consider a non-empty finite set $\Sigma$ and linearly independent functions $\pi_1,\dots,\pi_k:\Sigma\to\F_p$ that do not contain any non-zero constant function in their linear span. An important special case of this, which we shall need in Section \ref{Section: The general polynomial case}, is when $\Sigma$ is a subset $S\subset\F_p$ of size at least $k+1$, and $\pi_i(x)=x^i$ for $i=1,2,\dots,k$, since any non-trivial linear combination of the $\pi_i$ is then a polynomial of degree between 1 and $k$, which cannot be constant on a set of size greater than $k$. However, it will be convenient to us to prove our results in the more general set-up below.

If $\pi:\Sigma\to\F_p$, we shall also write $\pi^n:\Sigma^n\to\F_p^n$ for the map that applies $\pi$ pointwise, that is, for the map that takes $x=(x_1,\dots,x_n)$ to $\pi^n(x)=(\pi(x_1),\dots,\pi(x_n))$. And for any function $f: \Sigma^n \rightarrow \F_p$ and any $t\in \F_p^*$ we shall write $\bias_{t} f$ for the quantity $\E_{x \in \Sigma^n} \omega_p^{tf(x)}$.

Recall that we define the rank of a linear form to be the size of its support -- that is, for the number of its non-zero coefficients with respect to the standard basis.

\begin{proposition}\label{Equidistribution of multi-alphabet linear forms} Let $p$ be a prime, let $C_0 \ge 1$ be a positive integer, let $\Sigma$ be a non-empty set of size at most $C_0$, and let $\pi_1,\dots,\pi_k$: $ \Sigma \rightarrow \F_p$ be functions such that $\pi_1$ is not a linear combination of $\pi_2,\dots, \pi_k$ and a constant function. Let $ l_1,\dots,l_k: \F_p^n \rightarrow \F_p$ be linear forms. Then for each $ (a_1,\dots,a_k) \in \F_p^k$ with $ a_1 \neq 0$, and each $t \in \F_p^*$, 
\[ |\bias_{t} (a_1 l_1\circ\pi_1^n + \dots + a_k l_k\circ\pi_k^n)| \le (1-\pi^2/C_0p^2)^{\rk_1(l_1)}. \] 
\end{proposition}

\begin{proof} Let $(a_1,\dots,a_k) \in \F_p^k$ be such that $ a_1 \neq 0$ and let $t\in\F_p^*$. For each $1\leq i\leq k$ let the coefficients of $l_i$ be $l_{i1},\dots,l_{in}$ so $l_i(x)=\sum_{j=1}^nl_{ij}x_j$. Then
\[(a_1l_1\circ\pi_1^n+\dots+a_kl_k\circ\pi_k^n)(x)=\sum_{j=1}^n(a_1l_{1j}\pi_1(x_j)+\dots+a_kl_{kj}\pi_k(x_j)).\]
It follows that $\bias_{t} (a_1 l_1\circ\pi_1^n + \dots + a_k l_k\circ\pi_k^n)$ factors as
\[\prod_{j=1}^n\mathop{\E}_{x_j\in\Sigma}\omega_p^{t(a_1l_{1j}\pi_1(x_j)+\dots+a_kl_{kj}\pi_k(x_j))}.\]

Let $j$ be such that $l_{1j}$ is non-zero, and therefore such that $a_1l_{1j}$ is non-zero. Then by our assumption about the functions $\pi_i$, the function $a_1l_{1j}\pi_1 + \dots + a_k l_{kj} \pi_k$ 
is non-constant on $\Sigma$, which implies that it does not take any value with probability more than $1-C_0^{-1}$. By Lemma \ref{average of roots of unity}, it follows that 
\[ |\E_{x_j \in \Sigma} (\omega_p^{t(a_1l_{1j} \pi_1(x_j) + \dots + a_k l_{kj} \pi_k(x_j))})| \le 1-\pi^2/C_0p^2.\] 
Taking the product over all $j$ such that $l_{1j}\ne 0$, we conclude the desired inequality. \end{proof}

\begin{remark}\label{Low bias for polynomials without products of two variables} It follows from Proposition \ref{Equidistribution of multi-alphabet linear forms} that if $1 \le d \le p-1$ is a positive integer, $S$ is a subset of $\F_p$ with size at least $d+1$, and $P$ is a polynomial with degree $d$ of the type $ P = \sum_i P_i(x_i)$ for some polynomials $P_i:x\mapsto\sum_{j=0}^da_{ij}x_i^j$ of degree at most $d$ then 
\[|\bias_{t,S} P| \le (1-\pi^2/p^3)^{|\{i:a_{id}\ne 0\}|}\leq(1-\pi^2/p^3)^{\rk P-1}\] 
for every $t \in \F_p^*$. \end{remark}

We next generalize Proposition \ref{Equidistribution of multilinear forms for distributions}. Recall that if $\pi:\Sigma\to\F_p$ and $x\in\Sigma^n$, then we write $\pi^n(x)$ for $(\pi(x_1),\dots,\pi(x_n))$.

\begin{proposition}\label{Multi-alphabet projections with high rank} Let $p \ge 3$ be a prime, let $C_0, d, k$ and $l$ be positive integers with $d\geq 2$, let $\Sigma$ be a non-empty set of size at most $C_0$, and let $\pi_1,\dots,\pi_k: \Sigma \rightarrow \F_p$ be linearly independent maps that do not contain a non-zero constant map in their linear span. Let $ m_1,\dots, m_k: ((\F_p)^n)^d \rightarrow \F_p$ be $d$-linear forms, let $(a_1,\dots,a_k) \in \F_p^k$ be such that $a_1 \neq 0$ and let $ \epsilon > 0$ be a positive real number. If for a proportion at least $\epsilon$ of the $ x \in \Sigma^n$ the $(d-1)$-linear map
\[(y_2,\dots,y_d) \mapsto a_1 m_1(\pi_1^n(x),y_2,\dots,y_d) + \dots + a_k m_k(\pi_k^n(x),y_2,\dots,y_d)\] has partition rank at most $ l$, then 
\[\rk_1 m_1 \le 2(p-1) \log_{p/2}(p) Ml + (p-1)M \log_{p/2} \epsilon^{-1} \text{ if } d = 2 \text{ and }\]
\[\pr m_1 \le A_{d,\F_p}(2(p-1)Ml + (p-1)M \log_p \epsilon^{-1}) \text{ if } d \ge 3 \]
where $M=2kC_0p^2\log p/\pi^2$. \end{proposition}

\begin{proof}  Without loss of generality we can assume that $a_i \neq 0$ for each $ i \in \lbrack k \rbrack$: if this is not the case, then we proceed with a smaller $k$.

As earlier in the paper, given $ y^1 \in \F_p^n$ we write $(m_1)_{y^1}: (\F_p^n)^{d-1} \rightarrow \F_p^n$ for the $(d-1)$-linear form defined by the formula
\[(m_1)_{y^1} (y_2,\dots,y_d) = m_1(y^1,y_2,\dots,y_d).\] We shall carry out the proof in the case $d \ge 3$. In the case $d = 2$ the proof is the same except that we use the bound $\E_u p^{-\rk_1(m_u)} \le (2/p)^{\rk m}$ rather than $\E_u p^{-\ar m_u} = p^{-\ar m}$ .

For each $ y=(y^1,\dots,y^k) \in (\F_p^n)^k$, we shall also write $a.m_y: (\F_p^n)^{d-1} \rightarrow \F_p^n$ for the $(d-1)$-linear form defined by the formula
\[a.m_y(y_2,\dots,y_d) = a_1m_1(y^1,y_2,\dots,y_d) + \dots + a_km_k(y^k,y_2,\dots,y_d).\] 
Note that the $(d-1)$-linear map specified in the statement of the proposition is $a.m_{\pi^n(x)}$, where $\pi^n(x)$ is shorthand for $(\pi_1^n(x),\dots,\pi_k^n(x))$. 

Let $X=\{ x \in \Sigma^n: \pr(a.m_{\pi^n(x)}) \le l \}$ and suppose that $X$ has density at least $\epsilon$ inside $\Sigma^n$ with respect to the uniform probability measure on $\Sigma^n$. Write $\pi:\Sigma\to\F_p^k$ for the map $x\mapsto(\pi_1(x),\dots,\pi_k(x))$, and let $D$ be the measure on $\F_p^k$ defined by $D(B)=\P[\pi(x)\in B]$, where $x$ is chosen uniformly from $\Sigma$. Then the density of $\pi^n(X)$ with respect to the measure $D^n$ at least $\e$, since it is equal to $\P[\pi^n(x)\in\pi^n(X)]$, which is at least $\P[x\in X]$, which is at least $\e$ by hypothesis. Let $A=\pi^n(X)$. 

The statement that the $\pi_i$ are linearly independent and do not span a non-zero constant function can be expressed as follows: if $\lambda_1,\dots,\lambda_k$ and $w\in\F_p$ are such that $\sum_i\lambda_i\pi_i(u)=w$ for every $u\in\Sigma$, then $\lambda_1=\dots=\lambda_k=w=0$. This tells us that there is no proper affine subspace that contains all the functions $\psi_u:[k]\to\F_p$ defined by $\psi_u(i)=\pi_i(u)$.

In particular for any such subspace $W$, $D(W) \le 1-1/|\Sigma| \leq 1-1/C_0$, so applying Proposition \ref{Dense sumsets for distributions on $F_p^k$}, the set $ B=(p-1) MA$ has density at least $\epsilon^{(p-1)M}$ inside $ (\F_p^k)^n$. By averaging, there exists $ (y^2,\dots,y^k) \in (\F_p^n)^{k-1}$ such that the set 
\[Y^1=\{y^1 \in \F_p^n: (y^1,y^2,\dots,y^k) \in B\} \] 
has density at least $\epsilon^{(p-1)M}$ inside $ \F_p^n$. For each $ y^1 \in Y^1$ we have by subadditivity of the partition rank that $ \pr(a.m_{(y^1,\dots,y^k)}) \le (p-1)Ml$. Let $ y_0^1$ be a fixed element of $ Y^1$. For each $ y^1 \in Y^1$ the map $(m_1)_{y^1-y_0^1}$ can be rewritten as the difference $a_1^{-1}(a.m_{(y^1,y^2,\dots,y^k)} - a.m_{(y_0^1,y^2,\dots,y^k)})$, so $\pr(m_1)_{y^1-y_0^1} \le 2(p-1)Ml$ by subadditivity. By construction the set $ Y^1 - \{y_0^1\}$ has density at least $\epsilon^{(p-1)M}$ and for each $ y^1 \in Y^1 - \{y_0^1\}$, $ \pr m_{1}^{y^1} \le 2(p-1)Ml$. For each $y^1 \in Y^1 - \{y_0^1\}$, using the definition $\bias(m_1)_{y^1} = p^{-\ar(m_1)_{y^1}}$ of the analytic rank and Theorem 1.7 from \cite{Lovett}, which states that analytic rank is bounded above by partition rank, we have $\bias (m_1)_{y^1} \ge p^{-\pr(m_1)_{y^1}}$, so since $\pr(m_1)_{y^1} \le 2(p-1)Ml$ we obtain the lower bound $\bias(m_1)_{y^1} \ge p^{-2(p-1)Ml}$. 

We now use the fact that $\bias m_1 = \E_{y^1 \in \F_p^n} \bias(m_1)_{y^1}$. Since $Y^1$ has density at least $\epsilon^{(p-1)M}$ inside $\F_p^n$ and $\bias m_1^{y^1} > 0$ for for each $y_1 \in \F_p^n$ we obtain that $\bias m_1 \ge \epsilon^{(p-1)M} p^{-2(p-1)Ml}$. Therefore, 
\[ \pr m_1 \le A_{d,\F_p}(\ar m_1) \leq A_{d,\F_p}(2(p-1)Ml + (p-1)M \log_p \epsilon^{-1}) \] 
as desired. 
\end{proof}

We are now ready to prove a result that will have as a consequence that if the multilinear form associated with one ``piece" of a polynomial has high rank, then the whole polynomial has small bias. For fixed $p,k,C_0$ we define a sequence of functions $B_{d,p, k,C_0}$ by \[B_{1,p,k,C_0} = C_0p^2\log \epsilon^{-1} / \pi^2\] and for all $d \ge 2$, \[ B_{d,p, k,C_0} (\epsilon) = A_{d,\F_p}((p-1)(2kC_0p^2\log p/\pi^2)(2 B_{d-1,p, k, C_0}(\epsilon/2) + \log_p (\epsilon/2)^{-1})). \]

\begin{proposition} \label{Equidistribution of multi-alphabet multilinear forms}
Let $p \ge 3$ be a prime, let $C_0$, $d$ and $k$ be positive integers, let  $\Sigma$ be a non-empty set of size at most $C_0$, let $\pi_1,\dots,\pi_k: \Sigma \rightarrow \F_p$ be linearly independent maps that do not span a non-zero constant function, and let $\epsilon > 0$. For each $ (i_1,\dots,i_d) \in \lbrack k \rbrack^d$ let $ m_{(i_1,\dots,i_d)}: (\F_p^n)^d \rightarrow \F_p$ be a $d$-linear form. If there exists $ (i_1',\dots,i_d') \in \lbrack k \rbrack^d$ such that $ \pr m_{(i_1',\dots,i_d')} \ge B_{d,p, k, C_0} (\epsilon)$, then every linear combination $a.m^\pi: (\Sigma^n)^d \rightarrow \F_p$ defined by 
\[ a.m^\pi: (x_1,\dots,x_d) \mapsto \sum_{(i_1,\dots,i_d) \in \lbrack k \rbrack^d} a_{(i_1,\dots,i_d)} m_{(i_1,\dots,i_d)} (\pi_{i_1}^n(x_1), \dots, \pi_{i_d}^n(x_d)) \] 
with $ a_{(i_1',\dots,i_d')} \neq 0$ satisfies that
\[ |\bias_{t} a.m^\pi| \le \epsilon \] 
for all $t \in \F_p^*$. \end{proposition}

\begin{proof} We proceed by induction on $d$. The $d=1$ case holds by Proposition \ref{Equidistribution of multi-alphabet linear forms}. We now assume $d \ge 2$. Let $ a \in \F_p^{\lbrack k \rbrack^d}$ with $ a_{(i_1',\dots,i_d')} \neq 0$ be fixed throughout. Since $\pr m_{(i_1',\dots,i_d')} \ge B_{d,p, k, C_0}(\epsilon)$ and $ a_{(i_1',\dots,i_d')} \neq 0$, by Proposition \ref{Multi-alphabet projections with high rank} there exists a subset $X \subset \Sigma^n$ with density at most $\epsilon/2$ in $\Sigma^n$ and such that for all $x^1 \in \Sigma^n \setminus X$, the $ (d-1)$-linear form 
\[ (y_2,\dots,y_d) \mapsto \sum_{i_1=1}^k a_{(i_1,i_2',\dots,i_d')} m_{(i_1,i_2',\dots,i_d')}(\pi_{i_1}^n(x^1), y_2,\dots,y_d) \] 
has partition rank (or rather support size in the case $d=2$) at least $B_{d-1,p, k, C_0}(\epsilon/2)$. Let $t \in \F_p^*$ be fixed. For each $ x^1 \in \Sigma^n \setminus X$, applying Proposition \ref{Equidistribution of multi-alphabet multilinear forms} for $ d-1$ to the $(d-1)$-linear forms 
\[ (y_2,\dots,y_d) \mapsto \sum_{i_1=1}^k a_{(i_1,i_2,\dots,i_d)} m_{(i_1,i_2,\dots,i_d)}(\pi_{i_1}^n(x^1), y_2,\dots,y_d) \] 
with $(i_2,\dots,i_d) \in \lbrack k \rbrack^{d-1}$ we get 
\[ |\E_{(x_2, \dots, x_d) \in (\Sigma^n)^{d-1}} \omega_p^{ta.m^{\pi}(x^1, x_2 \dots, x_d)}| \le \epsilon/2. \] Because $X$ has density at most $\epsilon/2$ in $\Sigma^n$ we conclude that $|\bias_t a.m^{\pi}| \le \epsilon/2 + \epsilon/2 = \epsilon$. \end{proof}

\section{The general polynomial case} \label{Section: The general polynomial case}

Let $d \ge 1$ be a positive integer and let $P$ be a polynomial of degree exactly $d$. For a given number $ d'$ of pairwise distinct variables and for a given total degree $ t \in \{d',\dots,d\}$, let $ S(d',t)$ be the set of $ d'$-tuples of positive integers $ (s_1,\dots, s_{d'})$ with $ s_1 \ge s_2 \ge \dots \ge s_{d'} \ge 1$ and $ s_1 + \dots + s_{d'} = t$. We can decompose \begin{equation} P = \sum_{d'=0}^d \sum_{t=d'}^d \sum_{s \in S(d',t)} P_{s} \label{decomposition of a polynomial} \end{equation} where $ P_{s}$ is the part of $ P$ that consists of monomials of the type $x_{i_1}^{s_1}\dots x_{i_{d'}}^{s_{d'}}$ with $x_{i_1},\dots,x_{i_{d'}}$ distinct. We make the following definition.

\begin{definition}\label{Essential rank definition} The \emph{essential rank} of a part $P_{s}$, denoted by $\erk P_{s}$, is $ \min_Q \rk (P_{s} -V_s)$, where the minimum is taken over all polynomials $V_s$ that are linear combinations of monomials $x_{i_1}^{s_1}\dots x_{i_{d'}}^{s_{d'}}$ for which $x_{i_1},\dots,x_{i_{d'}}$ are not all distinct. 
\end{definition}

If $s=(s_1,\dots,s_{d'})$, then each $P_s$ can be written in the form
\[\sum_{i_1,\dots,i_{d'}}a_{i_1,\dots,i_{d'}}x_{i_1}^{s_1}\dots x_{i_{d'}}^{s_{d'}}\]
with $a_{i_1,\dots,i_{d'}}=0$ unless $i_1,\dots,i_{d'}$ are distinct. Let us partition the set $[d']$ into sets $I_1,\dots,I_r$ according to the value of $s_i$. Then for any permutation $\sigma$ of $[d']$ that leaves the sets $I_j$ invariant we have that $x_{i_1}^{s_1}\dots x_{i_{d'}}^{s_{d'}}=x_{i_{\sigma(1)}}^{s_1}\dots x_{i_{\sigma(d')}}^{s_{d'}}$, so if we replace each coefficient $a_{i_1,\dots,i_{d'}}$ by the average of the coefficients $a_{i_{\sigma(1)},\dots,i_{\sigma(d')}}$ over all such permutations, we obtain the same polynomial $P_s$, and now the coefficients have the symmetry property that $a_{i_1,\dots,i_{d'}}=a_{i_{\sigma(1)},\dots,i_{\sigma(d')}}$ whenever $\sigma$ is such a permutation. We therefore have a representation of $P_s$ in the form
\begin{equation}\label{Existence of an underlying multilinear form with symmetries} P_s(x)=m_s(x^{s_1},\dots,x^{s_{d'}}),\end{equation}
where $m_s$ is a $d$-linear form that is symmetric under all permutations of the variables that leave the sets $I_j$ invariant, and if $x=(x_1,\dots,x_n)$, then we write $x^{s_i}$ for the vector $(x_1^{s_i},\dots,x_n^{s_i})$. (It is not hard to show that this multilinear form is unique, using the fact that a non-zero polynomial of degree less than $p$ over $\F_p$ must take non-zero values, but we shall not need this.)

\begin{lemma}\label{Erk is at most epr} Let $s=(s_1,\dots,s_D)$ and suppose that $P_s\neq 0$. Then $\erk P_s \le \epr m_s$. \end{lemma}

\begin{proof} Assume that $\epr m_s \le k$ for some nonnegative integer $k$. Then there exists a $D$-linear form $m': (\F_p^n)^D \rightarrow \F_p$ such that the coefficient $m'_{(i_1,\dots,i_D)} = 0$ whenever $i_1,\dots,i_D \in \lbrack n \rbrack$ are distinct, and such that $\pr (m_s - m') \le k$. Then
\[(m_s-m')(y^{s_1},\dots,y^{s_D}) = m_s(y^{s_1},\dots,y^{s_D}) -m'(y^{s_1},\dots,y^{s_D}) \] 
for all $y \in\F_p^n$. The first term of the right-hand side is equal to $P_s(y)$, by the choice of $m_s$. The second term $m'(y^{s_1},\dots,y^{s_d})$ is a linear combination of monomials of the type 
$y_{i_1}^{s_1}\dots y_{i_D}^{s_D}$ with $i_1,\dots,i_D$ not distinct, so we can write it as $V_s(y)$ for some polynomial $V_s$ spanned by these monomials. It follows that $\erk P_s \le \rk (P_s-V_s)$.

Because $\pr(m_s - m') \le k$, for each $i \in \lbrack k \rbrack$ there exist a bipartition $\{J_i', J_i''\}$ of $\lbrack D \rbrack$ with $J_i', J_i''$ both non-empty and multilinear forms $M_{i,1}: (\F_p^n)^{J_i'} \rightarrow \F_p$, $M_{i,2}: (\F_p^n)^{J_i''} \rightarrow \F_p$  such that 
\[(m_s - m'): (z_1,\dots,z_D) \mapsto \sum_{i=1}^k M_{i,1}(z(J_i')) M_{i,2}(z(J_i'')).\] 
(Note that here $z_1,\dots,z_D$ are $D$ elements of $\F_p^n$ and $z=(z_1,\dots,z_D)$ is an element of $(\F_p^n)^D$.) For each $i \in \lbrack k \rbrack$ let the polynomials $Q_i, R_i$ be defined by 
\[ Q_i(y) = M_{i,1}(y^s(J_i'))\] 
and 
\[R_i(y) = M_{i,2}(y^s(J_i'')),\]
where we write $y^s$ as shorthand for $(y^{s_1},\dots,y^{s_D})$. Then for every $y \in \F_p^n$ we have 
\[ (P_s-V_s)(y) = \sum_{i=1}^k Q_i(y) R_i(y).\] 
Because $s_1, \dots, s_D \ge 1$, for any strict subset $J$ of $\lbrack D \rbrack$ we have $\sum_{j \in J} s_j < \sum_{j=1}^D s_j$, so $\deg Q_i, \deg R_i < \deg P$ for each $i \in \lbrack k \rbrack$. Therefore, $ \erk P \le k$. \end{proof}

Before starting the proof we note the following simple reduction for polynomials defined on restricted alphabets, which we shall use repeatedly.

\begin{lemma}\label{Replacing by a polynomial with monomials involving few pairwise distinct variables} Let $p$ be a prime, let $1 \le d \le p-1$ be a positive integer and let $S$ be a non-empty finite subset of $\F_p$. If $P$ is a polynomial of degree $d$ then $P$ coincides on $S^n$ with a linear combination of monomials of the type $\prod_{i} x_i^{s_i}$ with $s_i \le |S| - 1$ for all $i \in \lbrack n \rbrack$. \end{lemma}

\begin{proof} Whenever a monomial $\prod_{i=1}^n x_i^{s_i}$ contains a power $ x_i^{s_i}$ with $ s_i \ge |S|$ we can rewrite $ x^{s_i}$ as a linear combination of the $ x_i^{s_i'}$ with $ s_i' < s_i$, and hence rewrite the monomial. Each time a replacement is performed the difference between the previous monomial and the new monomial only takes the value $0$ on $S^n$. After all replacements we obtain a polynomial $ P - P_0$ which is spanned by the monomials $ \prod_{i} x_i^{s_i}$ with $ s_i \le (|S|) - 1$ for all $i \in \lbrack n \rbrack$ and which coincides with $P$ on $S^n$. \end{proof}

We note further that when we initially replace each monomial individually, every monomial from the new polynomial involves at most as many distinct variables as the monomial in the original polynomial did, and has degree at most that of the original monomial.

We are now ready to prove our main theorem.

\begin{proof}[Proof of Theorem \ref{Main theorem}]

Let $\epsilon > 0$ and let $ P$ be a polynomial of degree $d$ such that $|\bias_{t,S} P| \ge \epsilon$ for some $t \in \F_p^*$.

We first apply Lemma \ref{Replacing by a polynomial with monomials involving few pairwise distinct variables} to $P$, and, still writing $P$ for the resulting polynomial, we decompose $P$ into its pieces as in  \eqref{decomposition of a polynomial}, fix $D$ to be the highest value of $d' \in \lbrack d \rbrack$ such that there exist $t' \in \{d',\dots,d\}$, $s \in S(d',t')$ with $P_{s} \neq 0$, and set $T$ to be the largest such value of $t' \in \{d',\dots,d\}$. That is, $D$ is the largest number of variables involved in a monomial of $P$, and $T$ is the largest degree of a monomial that involves that number of variables (which is necessarily at least $D$ but may be less than $d$).

We prove the result by a double induction: the outer induction takes place on the degree $d$, and for a fixed $d$, we will use an inner induction with respect to the lexicographic order on the pairs $(D,T)$ with $1 \le D \le T \le d$. We will construct functions $(H_{p,d,S})_{\le (D,T)}: (0, 1 \rbrack \rightarrow \lbrack 0, \infty)$ such that if the relevant pair for $P$ is at most $(D,T)$ for this order and $|\bias_{t,S} P| \ge \epsilon$ for some $t \in \F_p^*$ and some $\epsilon > 0$ then $\rk_S P \le (H_{p,d,S})_{\le (D,T)}(\epsilon)$.

The base case of the induction is the case where $D=1$: We can write $P = \sum_{i=1}^n P_i(x_i)$ for some polynomials $P_i$ with degree at most $d$, and the result follows from the remark just before Proposition \ref{Multi-alphabet projections with high rank} with $(H_{p,d,S})_{\le (1,T)} = p^3\log \epsilon^{-1} /\pi^2 + 1$ for all $T \in \lbrack d \rbrack$.

Now let $D \ge 2$. We distinguish two cases. Let 
\[ \kappa_{p,D,d,S}(\epsilon)=\Lambda_D(B_{D,p,d,|S|^{2}}(\epsilon^{2^D})).\] 
Here $\Lambda_D$ is the function coming from Theorem \ref{Disjoint rank minors theorem}: if an order-$D$ tensor has essential partition rank at least $\Lambda_D(l)$ then it has disjoint partition rank at least $l$. As for $B_{D,p,d,|S|^{2}}$, it comes from Proposition \ref{Equidistribution of multi-alphabet multilinear forms}, the main result of the previous section. If there exists a polynomial $P_0: \F_p^n \rightarrow \F_p$ such that $P_0(S^n) = \{0\}$ and $\deg (P - P_0) < \deg P$ then we can conclude by the outer inductive hypothesis on $\deg P$, so we may assume without loss of generality that \begin{equation} \deg (P - P_0) \ge \deg P \label{degree assumption} \end{equation} for every polynomial $P_0: \F_p^n \rightarrow \F_p$ such that $P_0(S^n) = \{0\}$.
\medskip

\noindent \textbf{Case 1.} For each $ s \in S(D,T)$, we have $\erk P_{s} \le \kappa_{p,D,d,S}(\epsilon)$, and we can hence write

\[ P_{s} = V_{s} + \sum_{i=1}^{\kappa_{p,D,d,S}(\epsilon)} Q_{s,i} R_{s,i} \] 
where $ V_{s}$ is as in the definition of essential rank (Definition \ref{Essential rank definition}), and for each $ i \in \lbrack \kappa_{p,D,d,S}(\epsilon) \rbrack$, $ \deg Q_{s,i}$, $ \deg R_{s,i} \le T-1$. Moreover we can require that for each $i \in \lbrack \kappa_{p,D,d,S}(\epsilon) \rbrack$, all monomials of the polynomials $ Q_{s,i}$ and $ R_{s,i}$ involve at most $D$ pairwise distinct variables: if one of these polynomials, say $ Q_{s,i}$, contains a monomial with at least $D+1$ variables, then all the contributions of this monomial to $\sum_{i=1}^{\kappa_{p,D,d,S}(\epsilon)} Q_{s,i} R_{s,i}$ necessarily have to be cancelled by contributions from other $ Q_{s,i'} R_{s,i'}$ with $ i' \neq i$, as multiplication by any monomial other than $ 0$ cannot decrease the number of pairwise distinct variables in a monomial. Let 
\[P_{\new}= P - \sum_{s \in S(D,T)} (P_{s} - V_{s}) = \sum_{s \in S(D,T)} V_{s} + \sum_{{0 \le d' \le D, d' \le t' \le d}\atop{(d',t') \neq (D,T)}}\sum_{s\in S(d',t')} P_{s}. \]

Since $P_\new+\sum_{s\in S(D,T)}\sum_{i=1}^{\kappa_{p,D,d,S}(\e)}Q_{s,i}R_{s,i}=P$, and since $|\bias_{t,S}P|\geq\e$, Lemma \ref{Reduction to one function} implies that there exist $a_{s,i}, b_{s,i} \in \F_p$ for each $s\in S(D,T)$ and each $1\leq i \leq \kappa_{p,D,d,S}(\epsilon)$, such that the bias with respect to $t$ of the polynomial 
\[P_{\new}'= P_{\new} + \sum_{s \in S(D,T)} \sum_{i=1}^{\kappa_{p,D,d,S}(\epsilon)}(a_{s,i} Q_{s,i} + b_{s,i} R_{s,i}) \] 
is at least $p^{-2\kappa_{p,D,d,S}(\epsilon)} \epsilon$. The polynomial $P_{\new}'$ has the two following properties. 
\begin{enumerate} 
\item Each of its monomials has degree at most $T$ and also involves at most $D$ distinct variables. 
\item Each of its monomials has degree at most $T-1$, or involves at most $D-1$ distinct variables. 
\end{enumerate} 
The second property follows from the second expression for $P_\new$ given above, together with the fact that the monomials in $V_s$ involve fewer than $D$ distinct variables, and the fact that the monomials $Q_{s_i}$ and $R_{s,i}$ have degree less than $T$.

These two properties ensure that the pair $(D',T')$ associated with the polynomial $P_\new'$ is less than $(D,T)$ in lexicographical order, which will allow us to apply the inductive hypothesis. First, however, we apply Lemma \ref{Replacing by a polynomial with monomials involving few pairwise distinct variables} to $P_{\new}'$, obtaining a polynomial $P_{\new}''$ such that every monomial $\prod_{i=1}^n x_i^{s_i}$ of $P_{\new}''$ satisfies $s_i\leq |S| - 1$ for all $i \in \lbrack n \rbrack$, and such that the two properties above are still satisfied. Now using the inductive hypotheses we deduce that 
\begin{align*}
\rk_{S} P_{\new}'' &\le H_{p,d-1,S}(p^{-(2\kappa_{p,D,d,S}(\epsilon)+1)} \epsilon) \text{ if } \deg P_{\new}'' < d\\
\rk_{S} P_{\new}'' &\le (H_{p,d,S})_{\le (D,T-1)}(p^{-(2\kappa_{p,D,d,S}(\epsilon)+1)} \epsilon) \text{ if }\deg P_{\new}'' = d \text{ and } T > D\\
\rk_{S} P_{\new}'' &\le (H_{p,d,S})_{\le (D-1,d)}(p^{-(2\kappa_{p,D,d,S}(\epsilon)+1)} \epsilon) \text{ if }\deg P_{\new}'' = d \text{ and } T = D.\\
\end{align*} 
Using our assumption \eqref{degree assumption}, the fact that $P$ coincides with $P_{\new}'' + (P - P_{\new}) + (P_{\new} - P_{\new}')$ on $S^n$ (since $P'_\new$ and $P''_\new$ agree on $S^n$), the decomposition 
\[ P - P_{\new} = \sum_{s \in S(D,T)} \sum_{i=1}^{\kappa_{p,D,d,S}(\epsilon)} Q_{s,i} R_{s,i} \] 
and the fact that $P_{\new}' - P_{\new}$ is a linear combination of polynomials of degree strictly smaller than $\deg P$ we have 
\[\rk_S P \le (\rk_S P_{\new}'') + |S(D,T)| \kappa_{p,D,d,S}(\epsilon) + 1.\] 
It follows that $\rk_S P \le (H_{p,d,S})_{\le (D,S)}(\epsilon)$ with 
\begin{multline}(H_{p,d,S})_{\le (D,S)}= \max\Big\{H_{p,d-1,S}(p^{-(2\kappa_{p,D,d,S}(\epsilon)+1)} \epsilon), (H_{p,d,S})_{\le (D,T-1)}(p^{-(2\kappa_{p,D,d,S}(\epsilon)+1)} \epsilon),\\ 
(H_{p,d,S})_{\le (D-1,d)}(p^{-(2\kappa_{p,D,d,S}(\epsilon)+1)} \epsilon)\Big\} + |S(D,T)| \kappa_{p,D,d,S}(\epsilon) + 1.
\end{multline}
This concludes Case 1.
\medskip

We define the desired function $H_{p,d,S}$ to be $(H_{p,d,S})_{\le (d,d)}$. As we shall show, Case 2 will lead to a contradiction and this function is therefore suitable for Theorem \ref{Main theorem}.
\medskip

\noindent \textbf{Case 2.} There exists $s_0 \in S(D,T)$ such that 
\[\erk P_{s_0} \ge \Lambda_D(B_{D,p,d,|S|^{2}}(\epsilon^{2^D})).\] 
We start with the decomposition 
\begin{equation} P = \sum_{d'=0}^D \sum_{t'=d'}^d \sum_{s \in S(d',t')} P_{s}. \label{decomposition of P with respect to number of variables} \end{equation} 
For each $ t' \in \{D,\dots,d\}$ and $ s \in S(D,t')$ let $ m_{s}$ be a $D$-linear form of the form \eqref{Existence of an underlying multilinear form with symmetries}. That is, $P_s(x)=m_s(x^{s_1},\dots,x^{s_D})$, and $m$ is symmetric in $i$ and $j$ whenever $s_i=s_j$. 

By our assumption on $s_0$ and by Claim \ref{Erk is at most epr}, we have 
\[\epr m_{s_0} \ge \Lambda_D(B_{D,p,d,|S|^{2}}(\epsilon^{2^D})).\] 
By Theorem \ref{Disjoint rank minors theorem} applied to $ m_{s_0}$ there exist disjoint subsets $X_1,\dots,X_D \subset \lbrack n \rbrack$ such that 
\[\pr m_{s_0}(\F^{X_1} \times \dots \times \F^{X_D}) \ge B_{D,p,d,|S|^{2}}(\epsilon^{2^D}).\]

We now apply the argument from Proposition \ref{Decoupling inequality for polynomials}. Although we no longer obtain the inequality \eqref{decoupling inequality with $D-D$} (as we only know that all monomials of $P$ involve at most $D$ pairwise distinct variables rather than that all monomials of $P$ have degree at most $D$), using Proposition \ref{Decoupling for an arbitrary function} and following the first half of the proof of Proposition \ref{Decoupling inequality for polynomials} shows that
\begin{equation} 
|\E_{x \in S^n} \omega_p^{P(x)}|^{2^D} \le \E_{y_1, y_{-1} \in S^n} \omega_p^{\sum_{\nu \in \{-1,1\}^D} (-1)^{N(\nu)} P_{\{X_1,\dots,X_D\}}(y_{\nu_1}(X_1), \dots, y_{\nu_D}(X_D))}. \label{decoupling inequality in the arbitrary case} 
\end{equation} 
The polynomial $ P_{\{X_1,\dots,X_D\} }$ is equal to 
\begin{equation} 
\sum_{t'=D}^d \sum_{s \in S(D,t')} (P_{s})_{\{X_1,\dots,X_D\} } \label{Expression of the restricted polynomial} 
\end{equation} 
as the contribution of the terms from \eqref{decomposition of P with respect to number of variables} obtained from $ d' < D$ is zero. For a fixed $s$, we define an equivalence relation on $\mathcal S_D$, the set of permutations of $[D]$, by taking two permutations $ \sigma_1, \sigma_2$ to be equivalent if and only if $\sigma_2 \sigma_1^{-1}$ leaves the intervals $I_j$ (the intervals on which $s$ is constant) invariant. In other words, $\sigma_1$ and $\sigma_2$ are equivalent if the sequences $(s_{\sigma_1(1)},\dots,s_{\sigma_1(D)})$ and $(s_{\sigma_2(1)},\dots,s_{\sigma_2(D)})$ are equal. Let $ \mathcal{E}_s$ be the set of equivalence classes for this relation, and for each equivalence class $E \in \mathcal{E}_s$ let us pick a representative $\sigma_E \in E$.

For each $t' \in \{D,\dots,d\}$, $s \in S(D,t')$ and $y\in\F_p^n$, we have that $(P_{s})_{ \{ X_1,\dots,X_D \} }(y)$ is equal to
\begin{align*} & \sum_{\sigma \in \mathcal{S}_D} m_{s}(\F^{X_{\sigma(1)}} \times \dots \times \F^{X_{\sigma(D)}})(y(X_{\sigma(1)})^{s_1},\dots,y(X_{\sigma(D)})^{s_D})\\ 
= & \sum_{\sigma \in \mathcal{S}_D} m_{s}(\F^{X_1} \times \dots \times \F^{X_D})(y(X_1)^{s_{\sigma^{-1}(1)}},\dots,y(X_D)^{s_{\sigma^{-1}(D)}})\\
= & \sum_{\sigma \in \mathcal{S}_D} m_{s}(\F^{X_1} \times \dots \times \F^{X_D})(y(X_1)^{s_{\sigma(1)}},\dots,y(X_D)^{s_{\sigma(D)}})\\
= & \sum_{E \in \mathcal{E}_s} \sum_{\sigma \in E}m_{s}(\F^{X_1} \times \dots \times \F^{X_D})(y(X_1)^{s_{\sigma(1)}},\dots,y(X_D)^{s_{\sigma(D)}})\\
= & |I_1(s)|! \dots |I_r(s)|! \sum_{E \in \mathcal{E}_s}m_{s}(\F^{X_1} \times \dots \times \F^{X_D})(y(X_1)^{s_{\sigma_E(1)}},\dots,y(X_D)^{s_{\sigma_E(D)}})\\ \end{align*} using the symmetry of $ m_{s}$. For all $y \in \F_p^n$, using \eqref{Expression of the restricted polynomial} we obtain that $P_{\{X_1,\dots,X_D\} }(y)$ is equal to 
\begin{equation}\sum_{t'=D}^d \sum_{s \in S(D,t')} |I_1(s)|! \dots |I_r(s)|! \sum_{E \in \mathcal{E}_s}m_{s}(\F^{X_1} \times \dots \times \F^{X_D})(y(X_1)^{s_{\sigma_E(1)}},\dots,y(X_D)^{s_{\sigma_E(D)}}) \label{Computed expression of the restricted polynomial} \end{equation} 
and the exponent on the right-hand side of \eqref{decoupling inequality in the arbitrary case} can therefore be rewritten 
\begin{multline} \sum_{t'=D}^d \sum_{s \in S(D,t')} |I_1(s))|! \dots |I_r(s)|! \sum_{E \in \mathcal{E}_s}(m_{s})(\F^{X_1} \times \dots \times \F^{X_D})\\
(y_1(X_1)^{s_{\sigma_E(1)}} - y_{-1}(X_1)^{s_{\sigma_E(1)}},\dots, y_1(X_D)^{s_{\sigma_E(D)}} - y_{-1}(X_D)^{s_{\sigma_E(D)}}). \label{Computed expression of the exponent on the right-hand side of the decoupling inequality} \end{multline}

We finally apply the main result of the previous section, Proposition \ref{Equidistribution of multi-alphabet multilinear forms}. We apply it to the set $\Sigma = S^2$ and to the functions $\pi_i: \Sigma \rightarrow \F_p$ defined by $\pi_i(x',x'') = (x')^i - (x'')^i$ for $1\leq i \leq d-1$. These are linearly independent and do not span a non-zero constant function, as can be seen by fixing $x''$, and moreover we have $ \pr m_{s_0} \ge B_{D,p,d,|S|^{2}}(\epsilon^{2^D})$ and $|I_1(s_0)|! \dots |I_r(s_0)|! \neq 0$. Therefore, the assumptions of Proposition \ref{Equidistribution of multi-alphabet multilinear forms} is satisfied. Applying the proposition and \eqref{decoupling inequality in the arbitrary case} then shows that $|\bias_{t} P| < \epsilon$ for all $t \in \F_p^*$, which is incompatible with our assumption at the start of the proof. This finishes the proof. \end{proof}

\section{Surjectivity of multilinear forms on subsets of finite prime fields}

Proposition \ref{bounded analytic rank implies bounded partition rank for restricted alphabets} states that having a high partition rank is a sufficient condition for a multilinear form $m$ over $\F_p$ (of fixed order and for a fixed prime $p$) to be equidistributed on a product $S_1^n \times \dots \times S_d^n$ with $S_1, \dots, S_d$ subsets of $\F_p$ each containing at least two elements. In this section we show that to ensure that the restriction of $m$ to $S_1^n \times \dots S_d^n$ is surjective, it suffices to fulfill the qualitatively weaker condition that the multilinear form $m$ has high tensor rank.

\begin{definition}

Let $d \ge 2$ be a positive integer, let $\F$ be a field, and let $T: \lbrack n_1 \rbrack \times \dots \times \lbrack n_d \rbrack \rightarrow \F$. The \emph{tensor rank} of the tensor $T$, denoted by $\tr T$, is the smallest nonnegative integer $k$ such that there exist functions $a_{i,\a}: \lbrack n_{\a} \rbrack \rightarrow \F$ for all $\a \in \lbrack d \rbrack$ and all $i \in \lbrack k \rbrack$ such that we can write \[ T(x_1, \dots, x_d) = \sum_{i=1}^k a_{i,1}(x_1) \dots a_{i,d}(x_d) \] for every $(x_1, \dots, x_d) \in \lbrack n_1 \rbrack \times \dots \times \lbrack n_d \rbrack$. \end{definition}

We will use the following result which follows from repeatedly applying Proposition 11.4 from \cite{K.} (by performing the iterations in a way similar to those of the proof of Corollary 11.8 there). For $T$ an order-$d$ tensor, for $I$ a subset of $\lbrack d \rbrack$, for $y \in I^c$ and $z \in I$ let $T((y,z))$ be the value $T(x)$, where $x_{\a} = y_{\a}$ for all $\a \in I^c$ and $x_{\a} = z_{\a}$ for all $\a \in I$.

\begin{proposition}\label{Equivalence of tensor and partition rank} Let $d \ge 2$ be a positive integer, let $\F$ be a field, let $T: \lbrack n_1 \rbrack \times \dots \times \lbrack n_d \rbrack \rightarrow \F$, and let $l \ge 1$ be a positive integer. If $\pr T \le l$ and for every subset $I$ of $\lbrack d \rbrack$ with $2 \le |I| \le d-1$ and all $y \in \prod_{\a \in I^c} \lbrack n_{\a} \rbrack$, the order $|I|$ slice $T_y: x(I) \rightarrow T((x(I),y))$ has order $(d-|I|)$ partition rank at most $l$, then $\tr T \le (4l^3)^{2^d} $. \end{proposition}

We define a sequence $\Theta(p,d)$ for all $d \ge 2$ by \[\Theta(p,2) = K_{p,p^{-1},p^{-1}}((2p)^{-1}) \text{ and for all }d \ge 3\text{, }\Theta(p,d) = (4(A_{d,\F_p}(p \Theta(p,d-1)))^3)^{2^d},\] where $K_{p,p^{-1},p^{-1}}$ has been defined in Corollary \ref{flipped bilinear case}.

\begin{proposition}\label{Surjectivity of a multilinear form with high tensor rank} Let $ p$ be a prime and let $ d \ge 2$ be a positive integer. There exists $ \Theta(d,p)$ such that whenever $ T: \lbrack n_1 \rbrack \times \dots \times \lbrack n_d \rbrack \rightarrow \F_p$ is an order-$d$ tensor such that $ \tr T \ge \Theta(d,p)$, then whenever $ S_1,\dots,S_d$ are subsets of $ \F_p$ each containing at least two elements, the $d$-linear form $m:(\F_p^n)^d \rightarrow \F_p$ associated with $ T$ is surjective. \end{proposition}

\begin{proof} We prove the result by induction on $ d$. The result holds for $d=2$ by Proposition \ref{Equidistribution of high rank bilinear forms}. We now take $d \ge 3$, and assume $\tr T \ge \Theta(p,d)$.
\medskip

\noindent \textbf{Case 1}: There exists an order-$(d-1)$ slice of T, which without loss of generality we can assume to be the slice $T_a:(x_2,\dots,x_d) \mapsto T(a,x_2,\dots,x_d)$, with order-$(d-1)$ tensor rank at least $A_{d,\F_p}(p \Theta(p,d-1)) \ge p \Theta(p,d-1)$.

Because $S_1$ contains at least two elements, there exist $ b_1,\dots,b_p, c_1,\dots,c_p \in S_1$ such that $ b_1 + \dots + b_p = 1$ and $ c_1 + \dots + c_p = 0$. Writing the identity \[ 1_{a} = (b_1 1_{a} + c_1 (1-1_a)) + \dots + (b_p 1_a + c_p (1-1_a))\] between elements of $\F_p^{n_1}$ and using subadditivity of the tensor rank there exists $ i \in \lbrack p \rbrack$ such that $ \tr (b_i 1_a + c_i (1-1_a)) \ge \Theta(p,d-1)$. Letting $u$ the element of $\F_p^{n_1}$ defined by $ u:= b_i 1_a + c_i (1-1_a)$, by the inductive hypothesis we have $ m(\{u\} \times S_2^n \times \dots \times S_d^n) = \F_p$. By construction $ u \in S_1^n$, so in particular we have $ m(S_1^n \times S_2^n \times \dots \times S_d^n) = \F_p$.
\medskip

\noindent \textbf{Case 2}: We are not in Case 1. Then for every subset $I$ of $\lbrack d \rbrack$ with $2 \le |I| \le d-1$ and all $y \in \prod_{\a \in I^c} \lbrack n_{\a} \rbrack$, \[\pr T_y \le \tr T_y \le \tr T_{\{y_0\}} \le A_{d,\F_p}(p \Theta(p,d-1))\] where $T_{\{y_0\}}$ is an order $(d-1)$ slice of $T$ with domain containing the domain of $T_y$.

By our assumption $\tr T \ge \Theta(p,d)$ and Proposition \ref{Equivalence of tensor and partition rank} we necessarily have $ \pr T \ge A_{d,\F_p}(p \Theta(p,d-1))$. Therefore, $ \ar T \ge p \Theta(p,d-1)$. Defining for each $u \in \F_p^{n_1}$ the order $d-1$ tensor $u.T: \lbrack n_2 \rbrack \times \dots \times \lbrack n_d \rbrack \rightarrow \F_p$ by 
\[(u.T)(x_2, \dots, x_d) = \sum_{x_1=1}^{n_1} u(x_1) T(x_1,x_2, \dots, x_d)\] for every $(x_2, \dots, x_d) \in \lbrack n_2 \rbrack \times \dots \times \lbrack n_d \rbrack$ and using that $p^{-\ar T} = \E_{u \in \F_p^{n_1}} p^{-\ar u.T}$ we can find $u \in \F_p^{n_1}$ such that \[ \tr u.T \ge \pr u.T \ge \ar u.T \ge p \Theta(p,d-1).\]

The $ (d-1)$-linear form associated with the tensor $ u.T$ is $ (y_2,\dots,y_d) \mapsto m(u,y_2,\dots,y_d)$. Because $ S$ has size at least $ 2$, every element of $ \F_p$ can be written as a sum of at most $ p$ elements of $ S$, so for each $ x_1 \in \lbrack n \rbrack$ there exist $ u_1(x_1),\dots,u_p(x_1) \in S_1$ such that we can write $ u(x_1) = u_1(x_1) + \dots + u_p(x_1)$. By subadditivity $ \tr u.T \le \sum_{i=1}^p \tr u_i.T$, so there exists $i \in \lbrack p \rbrack$ such that $ \tr u_i.T \ge \Theta(p,d-1)$. Since $u_i \in S$ and by the inductive hypothesis we have $ m(\{u_i\} \times S_2^n \times \dots \times S_d^n) = \F_p$, we conclude that $m(S_1^n \times S_2^n \times \dots \times S_d^n) = \F_p$. \end{proof} 

In summary we have shown that for a fixed prime $p$, a fixed positive integer $d \ge 2$, and fixed non-empty subsets $S_1, \dots, S_d$ of $\F_p$, the behaviour of the range and distribution of a $d$-linear form $m:(\F_p)^d \rightarrow \F_p$ is as follows.

\begin{enumerate}

\item If $ \pr T$ is high (and hence $ \tr T$ is high) then $ m$ is approximately uniformly distributed on $S_1^n \times \dots \times S_d^n$.

\item If $ \pr T$ is low but $ \tr T$ is high, then $ m$ is not necessarily approximately uniformly distributed on $S_1^n \times \dots \times S_d^n$ but $m(S_1^n \times \dots \times S_d^n) = \F_p$.

\item If $ \tr T$ is low (and hence $ \pr T$ is low) then the image $m(S_1^n \times \dots \times S_d^n)$ is not necessarily the whole of $\F_p$.

\end{enumerate}

\section{Open problems}

Our results still leave open a number of questions. We have used the assumption $d < p$ in Theorem \ref{Main theorem} to obtain that $d! \neq 0$, which allowed us to assign a unique underlying $d$-linear form to a homogeneous polynomial of degree $d$, and later allowed us to more generally assign $D$-linear forms to the polynomials $P_{s}$. However it seems likely to us that this assumption can be removed. Indeed, the original paper \cite{Green and Tao} of Green and Tao used the assumption $d < p$ for the same purpose, to guarantee that $d! \neq 0$, but the assumption was nonetheless later removed in \cite{Kaufman and Lovett} by Kaufman and Lovett.

\begin{conjecture} Theorem \ref{Main theorem} still holds for $d \ge p$. \end{conjecture}

We can next ask try to improve our bounds. 

\begin{conjecture} \label{Bounds for multilinear forms} Let $p$ be a prime, let $d \ge 2$ be a positive integer, and let $S$ be a non-empty subset of $\F_p$. Then there exists a constant $C_{p,d,S}^{\pr}>0$ such that for every $d$-linear form $m:(\F_p^n)^d \rightarrow \F_p$, if $\max_{t \in\F_p^*} |\bias_{t,S} m| \ge \epsilon$ then $\pr m \le C_{p,d,S}^{\pr} \log \epsilon^{-1}$. \end{conjecture}

Conjecture \ref{Bounds for multilinear forms} specialises in the case $S =\F_p$ to the well-known conjecture that partition and analytic rank are equal up to a constant, and which was recently established in the large fields case by Cohen and Moshkovitz \cite{Cohen and Moshkovitz}. As we explained (following Janzer and Mili\'cevi\'c) in the introduction, in the case $S =\F_p$ a bound $\pr \le A_{d,p}(\ar)$ translates into a bound of the type $O_{p,d}(1) A_{d,p}(\ar)$ in Theorem \ref{Main theorem}. Although the proof of this implication no longer holds for an arbitrary non-empty subset $S$ of $\F_p$, it seems at least plausible to us that if Conjecture \ref{Bounds for multilinear forms} is true then we can take linear bounds in $\log \epsilon^{-1}$ in Theorem \ref{Main theorem}.

\begin{conjecture}\label{Bounds for polynomials} Let $p$ be a prime, let $d \ge 2$ be a positive integer, and let $S$ be a non-empty subset of $\F_p$. Then there exists a constant $C_{p,d,S}^{\rk}>0$ such that for every polynomial $P: \F_p^n \rightarrow \F_p$ with $\deg P = d$, if $\max_{t \in\F_p^*} |\bias_{t,S} P| \ge \epsilon$ then $\rk P \le C_{p,d,S}^{\rk} \log \epsilon^{-1}$. \end{conjecture}

Even if Conjecture \ref{Bounds for multilinear forms} is true, it would still not guarantee linear bounds in $\log \epsilon^{-1}$ in Theorem \ref{Main theorem}. It is conjectured in \cite{K.} that the disjoint partition rank and essential partition rank are also equal up to a constant. If this were proved then the bounds of Theorem \ref{Main theorem} would significantly improve, but owing to the inductive structure of our proof of Theorem \ref{Main theorem} this would still not yield linear bounds in $\log \epsilon^{-1}$. We therefore expect that arguments significantly different from those that we have used in our inductive proof would be required to prove Conjecture \ref{Bounds for polynomials}, if it is true.

In another direction we can ask for qualitative strengthenings of Theorem \ref{Main theorem}: even in the case $S = \F_p$, having high rank is merely a sufficient, and not a necessary condition for a polynomial to be approximately uniformly distributed. For instance, if $I,J$ are two disjoint subsets of $\lbrack n \rbrack$, and $P_1,P_2$ are two polynomials both with degree at least $2$ such that $P_1$ is a polynomial in the variables $x_i$, $i \in I$ and $P_2$ is a polynomial in the variables $x_j$, $j \in J$, then it suffices that the rank of either of the individual polynomials $P_1,P_2$ is large for their sum $P_1+P_2$ to be approximately uniformly distributed, but if $\deg P_1 < \deg P_2$ and $\rk P_2 =1$, then $\rk P_1+P_2 = 2$.

\begin{question} \label{Structure of polynomials} Let $p$ be a prime, let $d \ge 2$ be a positive integer, let $S$ be a non-empty subset of $\F_p$, and let $\epsilon > 0$. Let $P: \F_p^n \rightarrow \F_p$ be a polynomial such that $\max_{t \in\F_p^*} |\bias_{t,S} P| \ge \epsilon$. Can we describe the structure of $P$ over and above the fact that there exists a polynomial $P_0: \F_p^n \rightarrow \F_p$ such that $P_0(S^n) = \{0\}$ and $P-P_0$ has bounded rank ? \end{question}

In fact, to our knowledge not much is known about Question \ref{Structure of polynomials} even in the case $S=\F_p$. Theorem \ref{Main theorem} is an analogue of Proposition \ref{bounded analytic rank implies bounded partition rank for restricted alphabets} for polynomials, and we can also ask whether we can obtain an analogue of Proposition \ref{Surjectivity of a multilinear form with high tensor rank} for polynomials.

\begin{conjecture}\label{Range of polynomials} Let $p$ be a prime, let $d \ge 2$ be a positive integer, and let $S$ be a non-empty subset of $\F_p$. Then there exists a positive integer $R(p,d)$ such that if $P: \F_p^n \rightarrow \F_p$ is a polynomial and the restriction of $P$ to $S^n$ is not surjective, then there exists a polynomial $P_0: \F_p^n \rightarrow \F_p$ such that $P_0(S^n) = \{0\}$ and we can write \[P - P_0 = \sum_{i=1}^{R(p,d)} l_{i,1} \dots l_{i,d}\] for some linear forms $l_{i,j}: \F_p^n \rightarrow \F_p$. \end{conjecture}

Proposition \ref{Decoupling inequality for polynomials} no longer seems to be of any help for the task of proving Conjecture \ref{Range of polynomials}, but it is conceivable that some similar decoupling strategy could work to reduce Conjecture \ref{Range of polynomials} to Proposition \ref{Surjectivity of a multilinear form with high tensor rank}.


\begin{thebibliography}{9}

\bibitem{BDFKK}

J. Bourgain, S. J. Dilworth, K. Ford, S. Konyagin and D. Kutzarova, \textit{Explicit constructions of RIP matrices and related problems,} Duke Math. J. \textbf{159(1)} (2011), 145-185.  

\bibitem{Cohen and Moshkovitz}

A.Cohen and G.Moshkovitz, \textit{Partition rank and analytic rank are uniformly equivalent}, \url{https://arxiv.org/abs/2102.10509}.

\bibitem{DGIM}

J. de Dios Pont, R. Greenfeld, P. Ivanisvili and J. Madrid, \textit{Additive energies on discrete cubes}, \url{https://arxiv.org/abs/2112.09352}.

\bibitem{Gowers and K.}

W. T. Gowers and T. Karam, \textit{Modular obstructions to uniformity in the polynomial density Hales-Jewett conjecture}, in preparation.

\bibitem{Gowers and Wolf}

W. T. Gowers and J. Wolf, \textit{Linear forms and higher-degree uniformity for functions on} $\mathbb{F}_p^n$, Geom. Funct. Anal. \textbf{21} (2011), 36-69

\bibitem{Green and Tao}

B. Green and T. Tao, \textit{The distribution of polynomials over finite fields, with applications to the Gowers norms.} Contr. Discr. Math., \textbf{4} (2009), no. 2, 1-36.

\bibitem{Janzer}

O. Janzer, \textit{Polynomial bound for the partition rank vs the analytic rank of tensors}, Discrete Anal. \textbf{7} (2020), 1-18.

\bibitem{Kane and Tao}

D. Kane and T. Tao, \textit{A bound on partitioning clusters}, Elec. J. Combin. \textbf{24} (2017), \#~P2.31.

\bibitem{K.}

T. Karam, \textit{High-rank minors for high-rank tensors}, \url{https://arxiv.org/abs/2207.08030}.

\bibitem{Kaufman and Lovett}

T. Kaufman and S. Lovett, \textit{Worst case to average case reductions for polynomials}, 49th Annual IEEE Symposium on Foundations of Computer Science (2008), 166-175.

\bibitem{Kazhdan and Ziegler}

D. Kazhdan and T. Ziegler, \textit{Approximate cohomology}, Selecta Math. \textbf{24} (2018), 499-509.

\bibitem{Lovett}

S. Lovett, \textit{The analytic rank of tensors and its applications}, Discrete Anal. \textbf{7} (2019), 1-10.

\bibitem{Milicevic}

L. Mili\'cevi\'c, \textit{Polynomial bound for partition rank in terms of analytic rank}, Geom. Funct. Anal. \textbf{29} (2019), 1503-1530.

\bibitem{Naslund}

E. Naslund, \textit{The partition rank of a tensor and k-right corners in} $\mathbb{F}_q^n$, Jour. Combin. Th, A \textbf{174} (2020), 105190.




\end{thebibliography}
\end{document}